\documentclass[12pt]{amsart}

\usepackage{amssymb}
\usepackage{bm}
\usepackage{graphicx}
\usepackage{psfrag}
\usepackage{color}
\usepackage{url}
\usepackage{algpseudocode}
\usepackage{fancyhdr}
\usepackage{xy}
\usepackage{caption}
\usepackage{subcaption}
\input xy
\xyoption{all}

\newtheorem*{maintheorem*}{Main Theorem}
\newtheorem{theorem}{Theorem}[section]
\newtheorem{prop}[theorem]{Proposition}
\newtheorem{prob}[theorem]{Problem}
\newtheorem{lemma}[theorem]{Lemma}
\newtheorem{cor}[theorem]{Corollary}
\theoremstyle{definition}
\newtheorem{definition}[theorem]{Definition}
\newtheorem{example}[theorem]{Example}
\numberwithin{equation}{section}

\newcommand{\twopf}[4]
{
	\left\{
	\begin{array}{ll}
		#1 & \mbox{if } #2 \\
		#3 & \mbox{if } #4
	\end{array}
	\right.
}

\newcommand{\fivepf}[9]
{
	\left\{
	\begin{array}{ll}
		#1 & \mbox{if } #2 \\
		#3 & \mbox{if } #4 \\
		#5 & \mbox{if $i=1$,} \\
		#6 & \mbox{if } #7 \\
		#8 & \mbox{if } #9
	\end{array}
	\right.
}


\begin{document}
	
	\mbox{}
	\title{Dyck Paths and Positroids \\ from Unit Interval Orders}
	\author{Anastasia Chavez}
	\address{Department of Mathematics\\UC Berkeley\\Berkeley, CA 94720}
	\email{a.chavez@berkeley.edu}
	\author{Felix Gotti}
	\address{Department of Mathematics\\UC Berkeley\\Berkeley, CA 94720}
	\email{felixgotti@berkeley.edu}
	\date{\today}
	
	\begin{abstract}
		It is well known that the number of non-isomorphic unit interval orders on $[n]$ equals the $n$-th Catalan number. Using work of Skandera and Reed and work of Postnikov, we show that each unit interval order on $[n]$ naturally induces a rank $n$ positroid on $[2n]$. We call the positroids produced in this fashion \emph{unit interval positroids}. We characterize the unit interval positroids by describing their associated decorated permutations, showing that each one must be a $2n$-cycle encoding a Dyck path of length $2n$. We also provide recipes to read the decorated permutation of a unit interval positroid $P$ from both the antiadjacency matrix and the interval representation of the unit interval order inducing $P$. Using our characterization of the decorated permutation, we describe the Le-diagrams corresponding to unit interval positroids. In addition, we give a necessary and sufficient condition for two Grassmann cells parameterized by unit interval positroids to be adjacent inside the Grassmann cell complex. Finally, we propose a potential approach to find the $f$-vector of a unit interval order.
	\end{abstract}

\maketitle

\tableofcontents

\section{Introduction}

A \emph{unit interval order} is a partially ordered set that captures the order relations among a collection of unit intervals on the real line. Unit interval orders originated in the study of psychological preferences, first appearing in the work of Wiener \cite{nW14}, and then in greater detail in the work of Armstrong \cite{wA39} and others. They were also studied by  Luce~\cite{rL56} to axiomatize a class of utilities in the theory of preferences. Since then they have been systematically studied (see \cite{DK68, pF85, pF73, FM92, WF57, SR03} and references therein). These posets exhibit many interesting properties; for example, they can be characterized as the posets that are simultaneously $({\bf 3}+ {\bf 1})$-free and $({\bf 2} + {\bf 2})$-free (see Section~\ref{sec:background}). Moreover, it was first proved in \cite{WF57} that the number of non-isomorphic unit interval orders on $[n]$ equals $\frac{1}{n+1}\binom{2n}{n}$, the $n$-th Catalan number (see also \cite[Section~4]{DK68}).

In \cite{SR03}, motivated by the desire to understand the $f$-vectors of various classes of posets, Skandera and Reed showed that a simple procedure for labeling a unit interval order yields the useful form of its $n \times n$ antiadjacency matrix which is totally nonnegative (i.e., has all its minors nonnegative) with its zero entries appearing in a right-justified Young diagram located strictly above the main diagonal and anchored in the upper-right corner. The zero entries of such a matrix are separated from the one entries by a Dyck path joining the upper-left corner to the lower-right corner. Motivated by this observation, we call such matrices \emph{Dyck matrices}. The Hasse diagram and the antiadjacency (Dyck) matrix of a canonically labeled unit interval order are shown in Figure~\ref{fig:UIO and its antiadjacency matrix}.

\begin{figure}[h]
	\centering
	\includegraphics[width = 2.8cm]{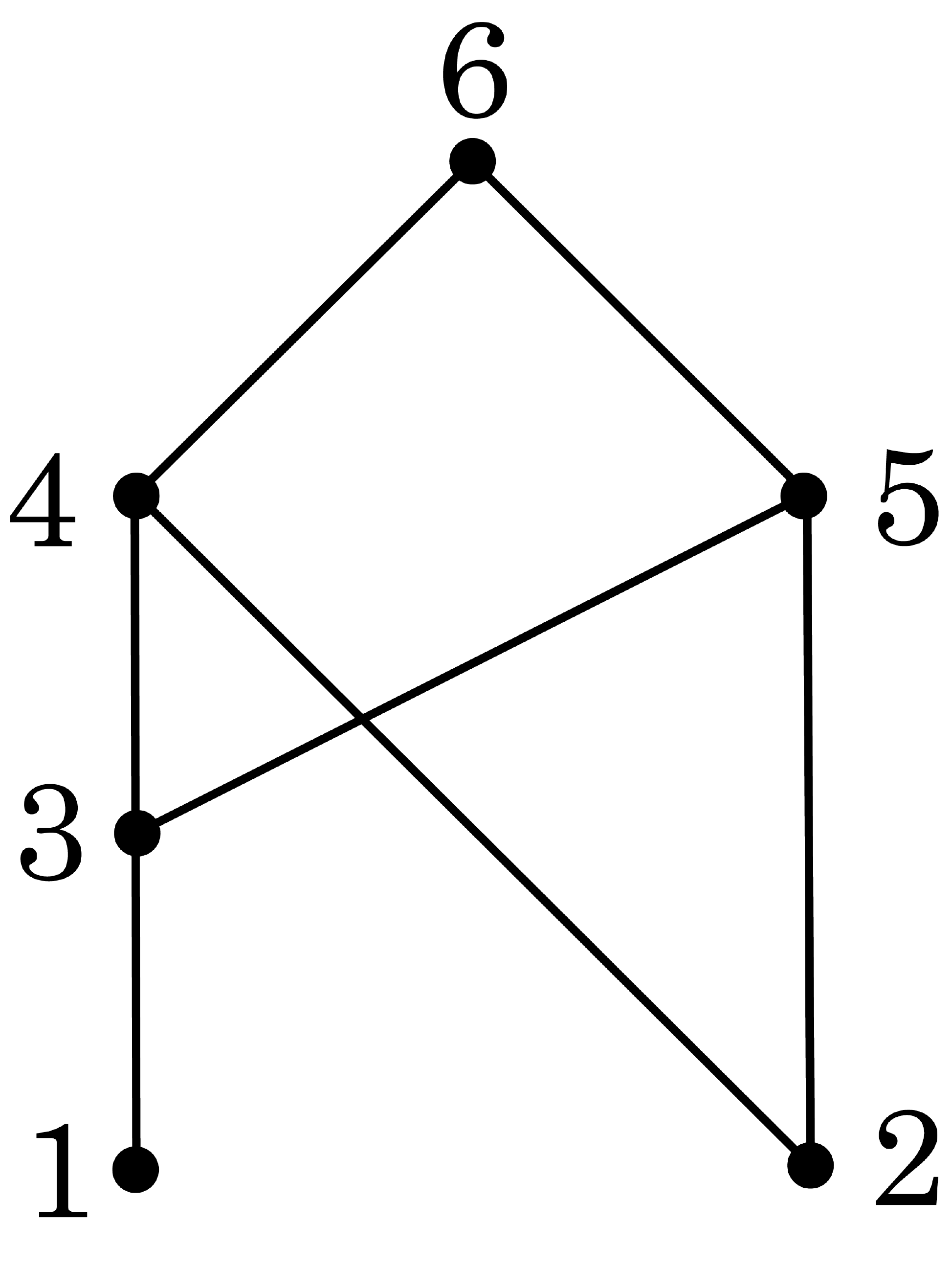} \quad \,
	\includegraphics[width = 1.2cm]{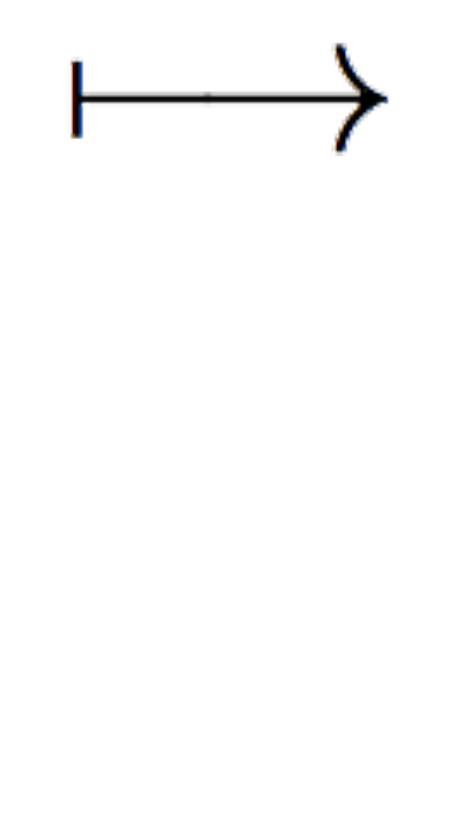} \quad \quad
	\includegraphics[width = 3.8cm]{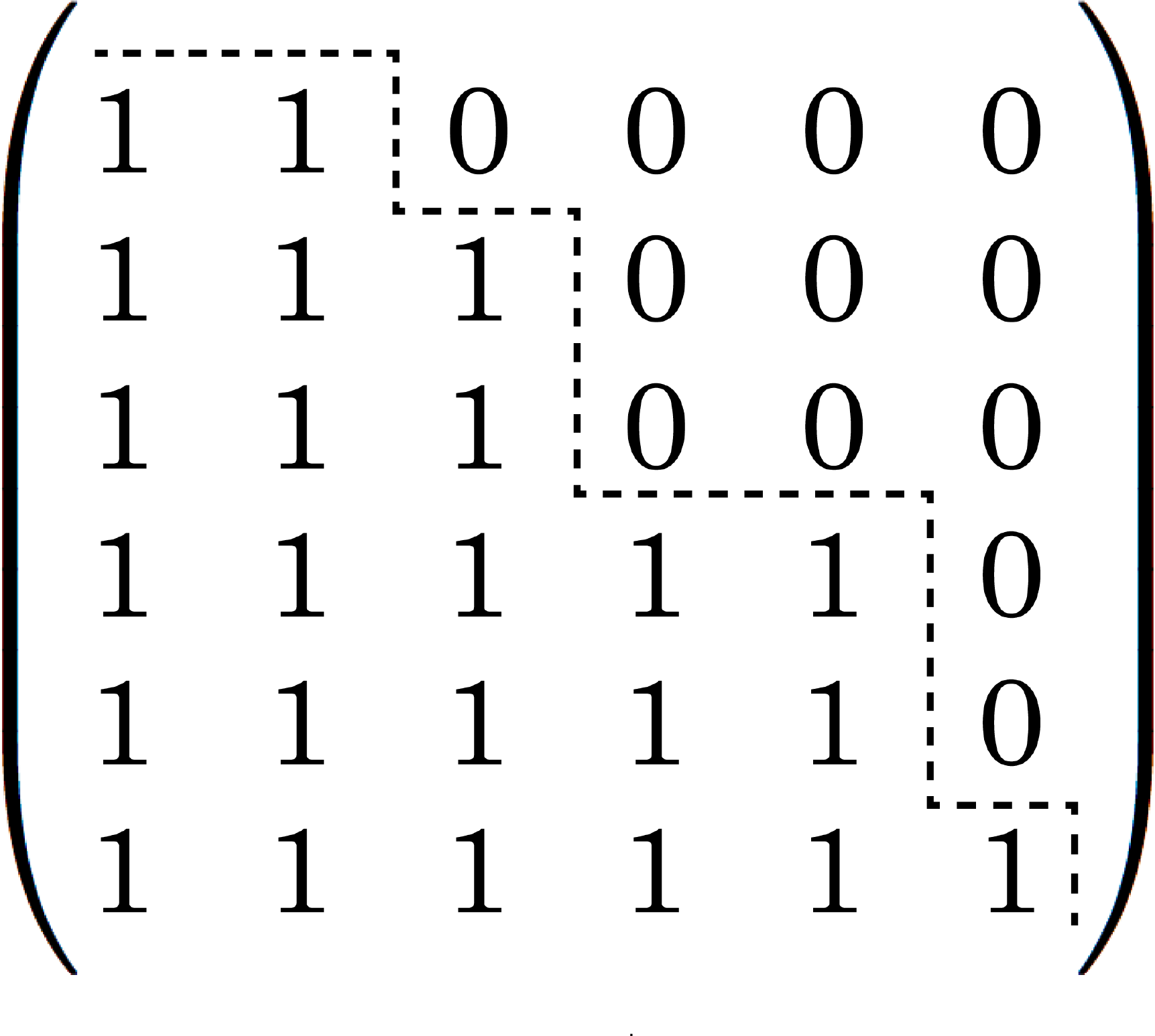}
	\caption{A canonically labeled unit interval order on $[6]$ and its antiadjacency matrix, in which one entries and zero entries are separated by a Dyck path.}
	\label{fig:UIO and its antiadjacency matrix}
\end{figure}

On the other hand, it follows from work of Postnikov \cite{aP06} that the $n \times n$ antiadjacency (Dyck) matrix of a (properly labeled) unit interval order $P$ can be regarded as representing a rank $n$ \emph{positroid} on the ground set $[2n]$. We will say that such a positroid is \emph{induced} by $P$. Positroids, which are special matroids, were introduced and classified by Postnikov in his study of the totally nonnegative part of the Grassmannian \cite{aP06}. He showed that there is a cell decomposition of the totally nonnegative part of the Grassmannian so that cells are indexed by positroids (or equivalent combinatorial objects).
%He showed that positroids are in bijection with various interesting families of combinatorial objects, including decorated permutations and Grassmann necklaces.
Positroids and the nonnegative Grassmannian have been the subject of a great deal of recent work, with connections and applications to cluster algebras \cite{jS06}, scattering amplitudes \cite{ABCGPT16}, soliton solutions to the Kadomtsev-Petviashvili equation \cite{KW14}, and free probability \cite{ARW16}.

In this paper we characterize the positroids that arise from unit interval orders, which we call \emph{unit interval positroids}.  We show that the decorated permutations associated to rank $n$ unit interval positroids are certain $2n$-cycles in bijection with Dyck paths of length $2n$. The following theorem is a formal statement of our main result (the necessary definitions are given in Section~\ref{sec:background}).

\begin{maintheorem*}
	A decorated permutation $\pi$ represents a unit interval positroid on $[2n]$ if and only if $\pi$ is a $2n$-cycle $(1 \ j_1 \ \dots \ j_{2n-1})$ satisfying the following two conditions:
	\begin{enumerate}
		\item in the sequence  $(1, j_1, \dots, j_{2n-1})$ the elements $1,\dots,n$ appear in increasing order while the elements $n+1, \dots, 2n$ appear in decreasing order;
		\item for every $1 \le k \le 2n-1$, the set $\{1, j_1, \dots, j_k\}$ contains at least as many elements of the set $\{1,\dots,n\}$ as elements of the set $\{n+1, \dots, 2n\}$.
	\end{enumerate}
	In particular, there are $\frac{1}{n+1} \binom{2n}{n}$ unit interval positroids on $[2n]$.
\end{maintheorem*}

The decorated permutation associated to a unit interval positroid on $[2n]$ induced by a unit interval order $P$ naturally encodes a Dyck path of length $2n$. Here we provide a recipe to read this decorated permutation directly from the Dyck path appearing in the antiadjacency (Dyck) matrix $A$ of $P$. We will refer to this path as the \emph{semiorder path} of $A$.

\begin{theorem} \label{thm:reading decorated permutation from Dyck matrix}
	Let $P$ be a canonically labeled unit interval order on $[n]$, and let $A$ be its antiadjacency matrix. Number the $n$ vertical steps of the semiorder path of $A$ from bottom to top by $1,\dots,n$ and label the $n$ horizontal steps from left to right by $n+1, \dots, 2n$. Then the sequence of $2n$ labels, read in the northwest direction, is the decorated permutation associated to the unit interval positroid induced by $P$.
\end{theorem}

\begin{example} \label{ex:main example}
	The vertical assignment on the left of Figure~\ref{fig:visual interpretation of our main result} shows a set $\mathcal{I}$ of unit intervals along with a canonically labeled unit interval order $P$ on $[5]$ describing the order relations among the intervals in $\mathcal{I}$ (see Theorem~\ref{thm:UIO characterization}). The vertical assignment on the right illustrates the recipe given in Theorem~\ref{thm:reading decorated permutation from Dyck matrix} to read the decorated permutation $\pi = (1 \ 2 \ 1\! 0 \ 3 \ 9 \ 4 \ 8 \ 7 \ 5 \ 6)$ associated to the unit interval positroid induced by $P$ directly from the antiadjacency matrix. Note that the decorated permutation $\pi$ is a $10$-cycle satisfying conditions (1) and (2) of our main theorem. The solid and dashed assignment signs represent functions that we shall introduce later.
%	\vspace{3pt}
	\begin{figure}[h]
		\centering
		\includegraphics[width = 15cm]{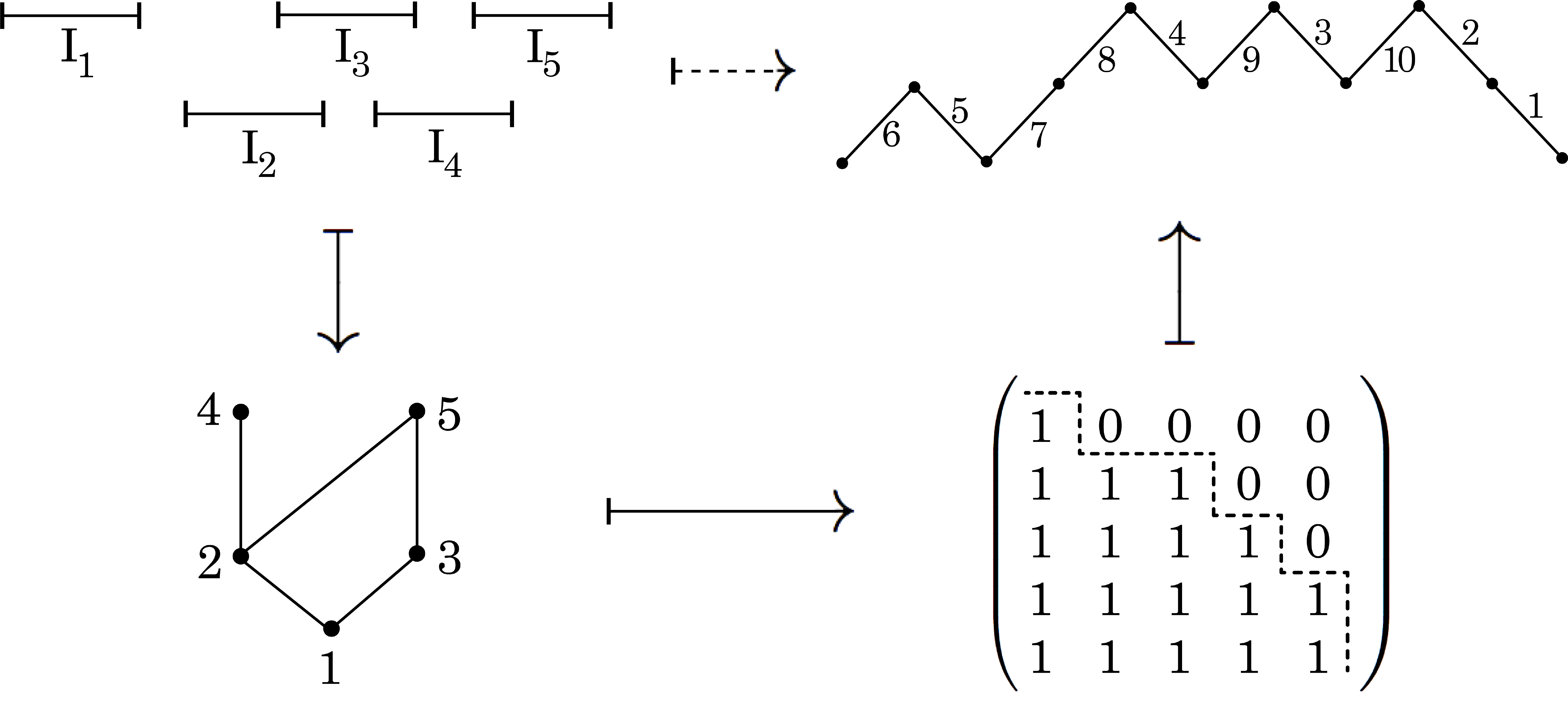}
		\caption{Following the solid assignments: unit interval representation $\mathcal{I}$, its unit interval order $P$, the antiadjacency matrix $\varphi(P)$, and the Dyck path that separates the one entries from the zero entries in $\varphi(P)$ showing the decorated permutation $\pi = (1 \ 2 \ 1\! 0 \ 3 \ 9 \ 4 \ 8 \ 7 \ 5 \ 6)$.}
		\label{fig:visual interpretation of our main result}
	\end{figure}
\end{example}

This paper is organized as follows. In Section~\ref{sec:background} we establish the notation and formally present the fundamental concepts and objects used throughout this paper. Then, in Section~\ref{sec:canonically labelings on UIO}, we formally introduce canonical labelings and canonical interval representations of unit interval orders. Also, we use canonical labelings to exhibit an explicit bijection from the set of non-isomorphic unit interval orders on $[n]$ to the set of $n \times n$ Dyck matrices. Section~\ref{sec:description of UIP} is dedicated to the description of the unit interval positroids via their decorated permutations, which yields the direct implication of the main theorem. In Section~\ref{sec:reading the UIP from the Dyck matrix}, we show how to read the decorated permutation associated to a unit interval positroid from either an antiadjacency matrix or a canonical interval representation of the corresponding unit interval order, which allows us to complete the proof of the main theorem. Then, in Section~\ref{sec:Le-diagram}, we characterize the Le-diagrams of unit interval positroids. In Section~\ref{sec:adjacency of UIP cells}, we find a necessary and sufficient condition for two unit interval positroids to index adjacent cells in the cell decomposition of the totally nonnegative part of the Grassmannian. Finally, in Section~\ref{sec:f-vector}, we interpret the $f$-vector of a poset in terms of its antiadjacency matrix and, based on this, we propose a potential approach to compute the $f$-vector of a unit interval order.

\section{Background and Notation} \label{sec:background}

For ease of notation, when $(P,\le_P)$ is a partially ordered set (\emph{poset} for short), we just write $P$, tacitly assuming that the order relation on $P$ is to be denoted by the symbol $\le_P$. For $x,y \in P$, we will write $x <_P y$ when $x \le_P y$ and $x \neq y$. In addition, every poset showing up in this paper is assumed to be finite.

\begin{definition}
	A poset $P$ is a \emph{unit interval order} if there exists a bijective map $i \mapsto [q_i, q_i+1]$ from $P$ to a set $S = \{[q_i, q_i + 1] \mid 1 \le i \le n, \, q_i \in \mathbb{R}\}$ of closed unit intervals of the real line such that for $i,j \in P$, $i <_P j$ if and only if $q_i + 1 < q_j$. We then say that $S$ is an \emph{interval representation} of $P$.
\end{definition}

For each $n \in \mathbb{N}$, we denote by $\mathcal{U}_n$ the set of all non-isomorphic unit interval orders of cardinality $n$. For nonnegative integers $n$ and $m$, let ${\bf n} + {\bf m}$ denote the poset which is the disjoint sum of an $n$-element chain and an $m$-element chain. Let $P$ and $Q$ be two posets. We say that $Q$ is an \emph{induced} subposet of $P$ if there exists an injective map $f \colon Q \to P$ such that for all $r,s \in Q$ one has $r \le_Q s$ if and only if $f(r) \le_P f(s)$. By contrast, $P$ is a $Q$-\emph{free} poset if $P$ does not contain any induced subposet isomorphic to $Q$. The following theorem provides a useful characterization of the elements of $\mathcal{U}_n$.

\begin{theorem} \cite[Theorem~2.1]{dS64} \label{thm:UIO characterization}
	A poset is a unit interval order if and only if it is simultaneously $(\bf{3}+\bf{1})$-free and $(\bf{2}+\bf{2})$-free.
\end{theorem}

If the poset $P$ has cardinality $n$, then a bijective function $\ell \colon P \to [n]$ is called an $n$-\emph{labeling} of $P$; after identifying $P$ with $[n]$ via $\ell$, we say that $P$ is an $n$-\emph{labeled} poset. The $n$-labeled poset $P$ is \emph{naturally labeled} if $i \le_P j$ implies that $i \le j$ as integers for all $i,j \in P$.

\begin{example}
	The figure below depicts the $6$-labeled unit interval order introduced in Figure~\ref{fig:UIO and its antiadjacency matrix} with a corresponding interval representation.
	\begin{figure}[h] \label{fig:UIO and interval representation}
		\centering
		\includegraphics[width = 2.8cm]{HasseDiagram6} \!\!
		\raisebox{-0.1\height}{\includegraphics[width = 1.2cm]{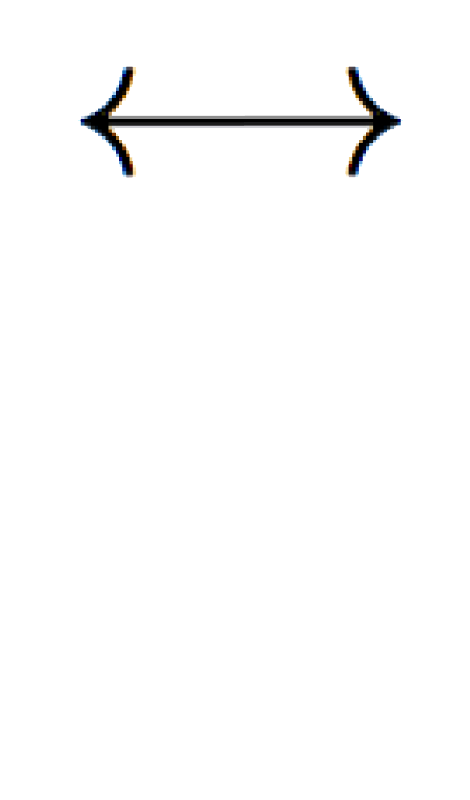}} \
		\includegraphics[width = 9.8cm]{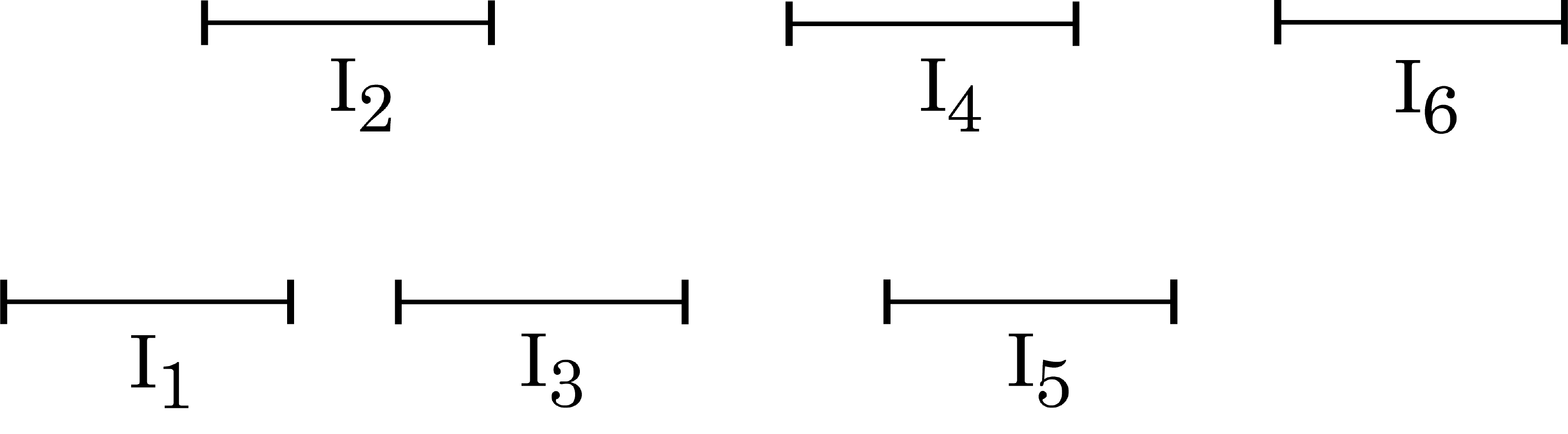}
		\caption{A $6$-labeled unit interval order and one of its interval representations.}
	\end{figure}
\end{example}

Another useful way of representing an $n$-labeled unit interval order is through its \emph{antiadjacency matrix}.

\begin{definition}
	If $P$ is an $n$-labeled poset, then the \emph{antiadjacency matrix} of $P$ is the $n \times n$ binary matrix $A = (a_{i,j})$ with $a_{i,j} = 0$ if and only if $i <_P j$.
\end{definition}

Recall that a binary square matrix $A$ is said to be a \emph{Dyck matrix} if its zero entries are separated from its one entries by a Dyck path joining the upper-left corner to the lower-right corner. We call such a Dyck path the \emph{semiorder path} of $A$. All minors of a Dyck matrix are nonnegative (see, for instance, \cite{ASW52}). We denote by $\mathcal{D}_n$ the set of all $n \times n$ Dyck matrices. As presented in \cite{SR03}, every unit interval order can be naturally labeled so that its antiadjacency matrix is a Dyck matrix (details provided in Section~\ref{sec:canonically labelings on UIO}). This yields a natural map $\varphi \colon \mathcal{U}_n \to \mathcal{D}_n$ that is a bijection (see Theorem~\ref{thm:bijection between UIOs and Dyck matrices}). In particular, $|\mathcal{D}_n|$ is the $n$-th Catalan number, which can also be deduced from the one-to-one correspondence between Dyck matrices and their semiorder (Dyck) paths.

For any matrix $A$ and any $k$-element subsets $I$, $J$ of row and column indices of $A$, define $\Delta_{I,J}(A)$, the $I$,$J$-\emph{minor of} $A$ to be the determinant of the submatrix of $A$ determined by rows $I$ and columns $J$. Let $\text{Mat}^{\ge 0}_{d,n}$ denote the set of all full rank $d \times n$ real matrices with nonnegative maximal minors. Given a totally nonnegative real $n \times n$ matrix $A$, there is a natural assignment $A \mapsto \psi(A)$, where $\psi(A) \in \text{Mat}^{\ge 0}_{n,2n}$.

\begin{lemma} \cite[Lemma~3.9]{aP06} \label{lem:correspondence between totally nonnegative square and rectangular matrix}\!\footnote{There is a typo in the entries of the matrix $B$ in \cite[Lemma~3.9]{aP06}.}
	For an $n \times n$ real matrix $A = (a_{i,j})$, consider the $n \times 2n$ matrix $B = \psi(A)$, where
	\[
		\begin{pmatrix}
			a_{1,1} 		& \dots		 & a_{1,n}     \\
			\vdots 	      & \ddots	   & \vdots     \\
			a_{n-1,1}    & \dots       & a_{n-1,n} \\
			a_{n,1}        & \dots      & a_{n,n}    
		\end{pmatrix} \
		\stackrel{\psi}{\longmapsto} \
		\begin{pmatrix}
			1  	  		&\dots   &  0        &   0  	 & (-1)^{n-1} a_{n,1} & \dots	& (-1)^{n-1} a_{n,n}  \\
			\vdots  &\ddots  &\vdots & \vdots & \vdots 	      			& \ddots & \vdots         			\\
			0  	  		&\dots   &  1        &   0  	 & -a_{2,1}  			& \dots   & -a_{2,n}    		    	\\
			0  	  		&\dots   &  0        &   1  	 & a_{1,1}        			& \dots   & a_{1,n}   
		\end{pmatrix}.
	\]
	For each pair $(I,J)$ with $I,J \subseteq [n]$ and $|I| = |J|$, define the set
	\[
		K = K(I,J) = \{n+1-k \mid k \in [n] \! \setminus \! I\} \cup \{n+j \mid j \in J\}. 
	\]
	Then we have $\Delta_{I,J}(A) = \Delta_{[n],K}(B)$.
\end{lemma}

Using Lemma~\ref{lem:correspondence between totally nonnegative square and rectangular matrix} and the aforementioned map $\varphi \colon \mathcal{U}_n \to \mathcal{D}_n$, we can assign via $\psi \circ \varphi$ a matrix of $\text{Mat}^{\ge 0}_{n,2n}$ to each unit interval order of cardinality $n$. In turn, every real matrix of $\text{Mat}^{\ge 0}_{n,2n}$ gives rise to a positroid, a special representable matroid which has a very rich combinatorial structure. Let us recall the concept of matroid.

\begin{definition}
	Let $E$ be a finite set, and let $\mathcal{B}$ be a nonempty collection of subsets of $E$. The pair $M = (E, \mathcal{B})$ is a \emph{matroid} if for all $B,B' \in \mathcal{B}$ and $b \in B \setminus B'$, there exists $b' \in B' \setminus B$ such that $(B \setminus \{b\}) \cup \{b'\} \in \mathcal{B}$.
\end{definition}

If $M = (E, \mathcal{B})$ is a matroid, then $E$ is called the \emph{ground set} of $M$ and the elements of $\mathcal{B}$ are called \emph{bases} of $M$. Any two bases of $M$ have the same size, which we denote by $r(M)$ and call the \emph{rank} of $M$. If $r(M) = d$ and $E = [n]$, then we say that $M$ is \emph{representable} if there exists a $d \times n$ real matrix $A$ with columns $A_1, \dots, A_n$ such that $B \in \mathcal{B}$ precisely when $\{A_b \mid b \in B\}$ is a basis for the vector space $\mathbb{R}^d$.

\begin{definition}
	The rank $d$ matroid on the ground set $[n]$ represented by a matrix $A \in \text{Mat}^{\ge 0}_{d,n}$ is denoted by $\rho(A)$ and called a \emph{positroid}.
\end{definition}

Each unit interval order $P$ (labeled so that its antiadjacency matrix is a Dyck matrix) induces a positroid via Lemma~\ref{lem:correspondence between totally nonnegative square and rectangular matrix}, namely, the positroid represented by the matrix $\psi(\varphi(P))$.

\begin{definition}
	A positroid on $[2n]$ induced by a unit interval order is called a \emph{unit interval positroid}.
\end{definition}

We denote by $\mathcal{P}_n$ the set of all unit interval positroids on the ground set $[2n]$. The function $\rho \circ \psi \circ \varphi \colon \mathcal{U}_n \to \mathcal{P}_n$ plays a fundamental role in this paper. Indeed, we will end up proving that this function is a bijection (see Theorem~\ref{thm:fundamental bijection}).

\vspace{4pt}
Several families of combinatorial objects, in bijection with positroids, were introduced in \cite{aP06} to study the totally nonnegative Grassmannian: decorated permutations, Grassmann necklaces, Le-diagrams, and plabic graphs. We use \emph{decorated permutations}, obtained from \emph{Grassmann necklaces}, to provide a compact and elegant description of unit interval positroids.

In the next definition subscripts should be interpreted modulo $n$.

\begin{definition}
	Let $d,n \in \mathbb{N}$ such that $d \le n$. An $n$-tuple $(I_1, \dots, I_n)$ of $d$-subsets of $[n]$ is called a \emph{Grassmann necklace} of type $(d,n)$ if for every $i \in [n]$ the following conditions hold:
	\begin{itemize}
		\item $i \in I_i$ implies $I_{i+1} = (I_i \setminus \{i\}) \cup \{j\}$ for some $j \in [n]$;
		\item $i \notin I_i$ implies $I_{i+1} = I_i$.
	\end{itemize}
\end{definition}

For $i \in [n]$, the total order $<_i$ on $[n]$ defined by $i <_i \dots <_i n <_i 1 <_i \dots <_i i-1$ is called the \emph{shifted linear $i$-order}. For a matroid $M = ([n], \mathcal{B})$ of rank $d$, one can define the sequence $\mathcal{I}(M) = (I_1, \dots, I_n)$, where $I_i$ is the lexicographically minimal ordered basis of $M$ with respect to the shifted linear $i$-order. It was proved in \cite[Section~16]{aP06} that the sequence $\mathcal{I}(M)$ is a Grassmann necklace of type $(d,n)$. We call $\mathcal{I}(M)$ the Grassmann necklace \emph{associated} to $M$. When $M$ is a positroid we can recover $M$ from its Grassmann necklace via Theorem~\ref{thm:bijection between positroids and GN}; however, this does not hold for a general matroid, which can be deduced also from Theorem~\ref{thm:bijection between positroids and GN}.

For $i \in [n]$, the \emph{Gale order} on $\binom{[n]}{d}$ with respect to $<_i$ is the partial order $\prec_i$ defined in the following way. If $S = \{s_1 <_i \dots <_i s_d\} \subseteq [n]$ and $T = \{t_1 <_i \dots <_i t_d\} \subseteq [n]$, then $S \prec_i T$ if and only if $s_j <_i t_j$ for each $j\in[d]$.

\begin{theorem}\cite[Theorem~6]{sO11} \label{thm:bijection between positroids and GN}
	For $d,n \in \mathbb{N}$ such that $d \le n$, let $\mathcal{I} = (I_1, \dots, I_n)$ be a Grassmann necklace of type $(d,n)$. Then
	\[
		\mathcal{B}(\mathcal{I}) = \bigg\{B \in \binom{[n]}{d} \ \bigg{|} \ I_j \prec_j B \ \text{for every} \ j \in [n] \bigg\}
	\]
	is the collection of bases of a positroid $M(\mathcal{I}) = ([n], \mathcal{B}(\mathcal{I}))$, where $\prec_i$ is the Gale $i$-order on $\binom{[n]}{d}$. Moreover, $M(\mathcal{I}(M)) = M$ for all positroids $M$.
\end{theorem}

Therefore there is a natural bijection between positroids on $[n]$ of rank $d$ and Grassmann necklaces of type $(d,n)$. However, \emph{decorated permutations}, also in one-to-one correspondence with positroids, will provide a more succinct representation.

\begin{definition}
	A \emph{decorated permutation} of $[n]$ is an element $\pi \in S_n$ whose fixed points $j$ are marked either ``clockwise" (denoted by $\pi(j)=\underline{j}$) or ``counterclockwise" (denoted by $\pi(j) = \overline{j}$).
\end{definition}

A \emph{weak $i$-excedance} of a decorated permutation $\pi \in S_n$ is an index $j \in [n]$ satisfying $j <_i \pi(j)$ or $\pi(j) = \overline{j}$. It is easy to see that the number of weak $i$-excedances does not depend on $i$, so we just call it the number of \emph{weak excedances}.

To every Grassmann necklace $\mathcal{I} = (I_1, \dots, I_n)$ one can associate a decorated permutation $\pi_{\mathcal{I}}$ as follows:
\begin{itemize}
	\item if $I_{i+1} = (I_i \setminus \{i\}) \cup \{j\}$, then $\pi_{\mathcal{I}}(j) = i$;
	\item if $I_{i+1} = I_i$ and $i \notin I_i$, then $\pi_\mathcal{I}(i) = \underline{i}$;
	\item if $I_{i+1} = I_i$ and $i \in I_i$, then $\pi_\mathcal{I}(i) = \overline{i}$.
\end{itemize}
The assignment $\mathcal{I} \mapsto \pi_{\mathcal{I}}$ defines a one-to-one correspondence between the set of Grassmann necklaces of type $(d,n)$ and the set of decorated permutations of $[n]$ having exactly $d$ weak excedances.

\begin{prop}\cite[Proposition 4.6]{ARW16}
	The map $\mathcal{I} \mapsto \pi_{\mathcal{I}}$ is a bijection between the set of Grassmann necklaces of type $(d,n)$ and the set of decorated permutations of $[n]$ having exactly $d$ weak excedances.
\end{prop}

\begin{definition}
	If $P$ is a positroid and $\mathcal{I}$ is the Grassmann necklace associated to $P$, then we call $\pi_{\mathcal{I}}$ the decorated permutation \emph{associated} to $P$.
\end{definition}

\section{Canonical Labelings on Unit Interval Orders} \label{sec:canonically labelings on UIO}

In this section we introduce the concept of a \emph{canonically} labeled poset, and we use it to exhibit an explicit bijection from the set $\mathcal{U}_n$ of non-isomorphic unit interval orders of cardinality $n$ to the set $\mathcal{D}_n$ of $n \times n$ Dyck matrices.

Given a poset $P$ and $i \in P$, the \emph{order ideal} and the \emph{dual order ideal} of $i$ are defined to be $\Lambda_i = \{j \in P \mid j \le_P i\}$ and $\text{V}_i = \{j \in P \mid i \le_P j\}$, respectively. We define the \emph{altitude} function of $P$ to be the map $\alpha \colon P \to \mathbb{Z}$ defined by $i \mapsto |\Lambda_i| - |\text{V}_i|$. We say that an $n$-labeled poset $P$ \emph{respects} altitude if for all $i,j \in P$, the fact that $\alpha(i) < \alpha(j)$ implies $i < j$ (as integers). Notice that every poset can be labeled by the set $[n]$ such that, as an $n$-labeled poset, it respects altitude (see also \cite[p.~33]{pF85}).

\begin{definition}
	An $n$-labeled poset is \emph{canonically labeled} if it respects altitude.
\end{definition}

Each canonically $n$-labeled poset is, in particular, naturally labeled. The next proposition, extending \cite[proof of Theorem~2.11]{wT92}, characterizes canonically $n$-labeled unit interval orders in terms of their antiadjacency matrices.
\vspace{-2pt}
\begin{prop}\cite[Proposition~5]{SR03} \label{prop:a labeled UIO is canonical iff its antiadjacency matrix is Dyck}
	An $n$-labeled unit interval order is canonically labeled if and only if its antiadjacency matrix is a Dyck matrix.
\end{prop}
\vspace{-2pt}
The above proposition indicates that the antiadjacency matrices of canonically labeled unit interval orders are quite special. In addition, canonically labeled unit interval orders have very convenient interval representations.

\begin{prop} \label{prop:interval representations of a canonically labeled UIO}
	Let $P$ be an $n$-labeled unit interval order. Then the labeling of $P$ is canonical if and only if there exists an interval representation $\{[q_i, q_i + 1] \mid 1 \le i \le n\}$ of $P$ such that $q_1 < \dots < q_n$.
\end{prop}

\begin{proof}
	Let $\alpha \colon P \to \mathbb{Z}$ be the altitude map of $P$. For the forward implication, suppose that the $n$-labeling of $P$ is canonical. Among all interval representations of $P$, assume that $\{[q_i, q_i + 1] \mid 1 \le i \le n\}$ gives the maximum $m \in [n]$ such that $q_1 < \dots < q_m$. Suppose, by way of contradiction, that $m < n$. The maximality of $m$ implies that $q_m > q_{m+1}$. This, along with the fact that $\alpha(m) \le \alpha(m+1)$, ensures that $q_m \in (q_{m+1}, q_{m+1} + 1)$. Similarly, $q_i + 1 \notin (q_{m+1}, q_m)$ for any $i \in [n]$; otherwise
	\[
		\alpha(m+1) = |\Lambda_{m+1}| - |\text{V}_{m+1}| < |\Lambda_m| - |\text{V}_{m+1}| \le |\Lambda_m| - |\text{V}_m| = \alpha(m)
	\]
	would contradict that the $n$-labeling of $P$ respects altitude. An analogous argument guarantees that  $q_i \notin (q_{m+1} + 1, q_m + 1)$ for any $i \in [n]$.
	
	Now take $k$ to be the smallest natural number in $[m]$ such that $q_j > q_{m+1}$ for all $j \ge k$, and take $\sigma = (k \, \ k\!+\!1 \ \dots \ m \, \ m\!+\!1) \in S_n$. We will show that $S = \{[p_i, p_i + 1] \mid 1 \le i \le n\}$, where $p_i = q_{\sigma(i)}$, is an interval representation of $P$. Take $i, j \in P$ such that $i \le_P j$. Since $i$ and $j$ are comparable in $P$, at least one of them must be fixed by $\sigma$; say $\sigma(i) = i$. If $\sigma(j) = j$, then $p_i + 1 = q_i + 1 < q_j = p_j$. Also, if $\sigma(j) \neq j$, then $q_i + 1 < q_j \in (q_{m+1}, q_m)$. It follows from $q_i + 1 < q_m$ that $p_i + 1 = q_i + 1 < q_{m+1} < q_{\sigma(j)} = p_j$. The case of $\sigma(j) = j$ can be argued similarly. Thus, $S$ is an interval representation of $P$. As $q_1 < \dots < q_m$, the definition of $k$ implies that $p_1 < \dots < p_{m+1}$, which contradicts the maximality of $m$. Hence $m=n$, and the direct implication follows.
	
	Conversely, note that if $\{[q_i, q_i + 1] \mid 1 \le i \le n\}$ is an interval representation of $P$ satisfying $q_1 < \dots < q_n$, then for every $m \in [n-1]$ we have
	\[
		\alpha(m) = |\Lambda_m| - |\text{V}_m| \le |\Lambda_{m+1}| - |\text{V}_{m+1}| = \alpha(m+1),
	\]
	which means that the labeling of $P$ is canonical.
\end{proof}

If $P$ is a canonically $n$-labeled unit interval order, and $\mathcal{I} = \{[q_i,q_i + 1] \mid 1 \le i \le n\}$ is an interval representation of $P$ satisfying $q_1 < \dots < q_n$, then we say that $\mathcal{I}$ is a \emph{canonical} interval representation of $P$.

Note that the image (as a multiset) of the altitude map does not depend on the labels but only on the isomorphism class of a poset. On the other hand, the altitude map $\alpha_P$ of a canonically $n$-labeled unit interval order $P$ satisfies $\alpha_P(1) \le \dots \le \alpha_P(n)$. Thus, if $Q$ is a canonically $n$-labeled unit interval order isomorphic to $P$, then
\begin{equation} \label{eq:equality of altitude vectors}
	(\alpha_P(1), \dots, \alpha_P(n)) = (\alpha_Q(1), \dots, \alpha_Q(n)),
\end{equation}
where $\alpha_Q$ is the altitude map of $Q$. Let $A_P$ and $A_Q$ be the antiadjacency matrices of $P$ and $Q$, respectively. As $\alpha_P(1) = \alpha_Q(1)$, the first rows of $A_P$ and $A_Q$ are equal. Since the number of zeros in the $i$-th column (resp., $i$-th row) of $A_P$ is precisely $|\text{V}_i(P) - 1|$ (resp., $|\Lambda_i(P)| - 1$), and similar statement holds for $Q$, the next lemma follows immediately by using \eqref{eq:equality of altitude vectors} and induction on the row index of $A_P$ and $A_Q$.

\begin{lemma} \label{lem:isomorphic canonically labeled UIO have equal Dyck matrix}
	If two canonically labeled unit interval orders are isomorphic, then they have the same antiadjacency matrix.
\end{lemma}

Now we can define a map $\varphi \colon \mathcal{U}_n \to \mathcal{D}_n$, by assigning to each unit interval order its antiadjacency matrix with respect to any of its canonical labelings. By Lemma~\ref{lem:isomorphic canonically labeled UIO have equal Dyck matrix}, this map is well defined.

\begin{theorem} \label{thm:bijection between UIOs and Dyck matrices}
	For each natural number $n$, the map $\varphi \colon \mathcal{U}_n \to \mathcal{D}_n$ is a bijection.
\end{theorem}

\begin{proof}
	It follows by combining the proof of \cite[Ch.~8, Proof of Thm.~2.11]{wT92} and \cite[Prop.~5]{SR03}.
\end{proof}

\section{Description of Unit Interval Positroids} \label{sec:description of UIP}

 We proceed to describe the decorated permutation associated to a unit interval positroid. Throughout this section $A$ is an $n \times n$ Dyck matrix and $B = (b_{i,j}) = \psi(A)$ is as in Lemma~\ref{lem:correspondence between totally nonnegative square and rectangular matrix}. We will consider the indices of the columns of $B$ modulo $2n$. Furthermore, let $P$ be the unit interval positroid represented by $B$, and let $\mathcal{I}_P$ and $\pi^{-1}$ be the Grassmann necklace and the decorated permutation associated to $P$.

\begin{lemma} \label{lem:removing a column of B does not affect its rank}
	For $1 < i \le 2n$, the $i$-th coordinate set of $\mathcal{I}_P$ does not contain $i-1$.
\end{lemma}

\begin{proof}
	It is not hard to verify that every matrix resulting from removing one column from $B$ still has rank $n$. As the matrix obtained by removing the $(i-1)$-st column from $B$ has rank $n$, it contains $n$ linearly independent columns. Therefore the lemma follows straightforwardly from the $<_i$-minimality of the $i$-th coordinate set of $\mathcal{I}_P$.
\end{proof}

For the rest of this section let $B_j$ denote the $j$-th column of $B$. As a direct consequence of Lemma~\ref{lem:removing a column of B does not affect its rank}, we have that $\pi$ does not have any counterclockwise fixed point. On the other hand, $\pi$ cannot have any clockwise fixed point because every column of $B$ is nonzero. Hence $\pi$ (and therefore $\pi^{-1}$) does not fix any point. The next lemma immediately follows from the way $\pi^{-1}$ is produced from the Grassmann necklace $\mathcal{I}_P$ (see the end of Section~\ref{sec:background}).

\begin{lemma} \label{lem:decorated permutation image}
	For $i \in \{1,\dots,2n\}$, $\pi(i)$ equals the minimum $j \in [2n]$ with respect to the $i$-order such that $B_i \in \emph{span}(B_{i+1},\dots, B_j)$.
\end{lemma}

Now we find an explicit expression for the function representing the inverse $\pi$ of the decorated permutation $\pi^{-1}$ associated to $P$. In order to do so, we will find it convenient to associate an index set and a map to the matrix $B$. We define the \emph{set of principal indices} of $B$ to be the subset of $\{n+1, \dots, 2n\}$ defined by
\[
	J = \{j \in \{n+1, \dots, 2n\} \mid B_j \neq B_{j-1}\}.
\]
We associate to $B$ the \emph{weight} map $\omega \colon [2n] \to [n]$ defined by $\omega(j) = \max\{i \mid b_{i,j} \neq 0\}$; more explicitly, we obtain that
\[
	\omega(j) = \twopf{j}{j \in \{1,\dots,n\}}{|b_{1,j}| + \dots + |b_{n,j}|}{j \in \{n+1, \dots, 2n\}.}
\]
Since the last row of the antiadjacency matrix $A$ has all its entries equal to $1$, the map $\omega$ is well defined. If $j \in \{n+1,\dots, 2n\}$, then $\omega(j)$ is the number of nonzero entries in the column~$B_j$. Now we have the following formula for $\pi$.

\begin{prop} \label{prop:explicity function for decorated permutations}
	 For $i \in \{1,\dots,2n\}$, we have
	\[
		\pi(i) = \fivepf{i+1}{n < i < 2n \text{ and } i+1 \notin J,}{\omega(i)}{n < i < 2n \text{ and } i+1 \in J, \text{ or } i=2n,}{n+1}{i-1}{1 < i \le n \text{ and } i-1 \notin \omega(J),}{j}{1 < i \le n \text{ and } i-1 = \omega(j) \text{ for some } j \in J.}
	\]
	The index $j$ in the final case is necessarily unique.
\end{prop}

\begin{proof}
	First, suppose that $n < i < 2n$ and $i+1 \notin J$. Then we have $B_i = B_{i+1}$ and the set $\{B_i,B_{i+1}\}$ is linearly dependent. Lemma~\ref{lem:decorated permutation image} then implies that $\pi(i) = i+1$.
	
	Now suppose that $n < i < 2n$ and $i+1 \in J$. Then $B_{i+1}$ results from replacing $m$ ($m > 0$) of the last nonzero entries of $B_i$ by zeros. Since $i+1 \in J$, the indices $i$ and $i+1$ both appear in the $i$-th coordinate set of $\mathcal{I}_P$. Also, because the columns $B_i, B_{i+1}, B_{\omega(i+1) + 1}, \dots, B_{\omega(i)}$ are linearly dependent, not all the indices $\omega(i+1) + 1, \dots, \omega(i)$ appear in the $i$-th coordinate set of $\mathcal{I}_P$. On the other hand, at most one index in $\omega(i+1) + 1, \dots, \omega(i)$ is missing from the $i$-th coordinate of $\mathcal{I}_P$; this is because the submatrix of $B$ determined by the row-index set $\{\omega(i+1) + 1, \dots, \omega(i)\}$ and the column-index set $\{n+1, \dots, 2n\}$ has rank $1$. By the minimality of the $i$-th coordinate set of $\mathcal{I}_P$ with respect to the $i$-order, the index of $\{\omega(i+1) + 1, \dots, \omega(i)\}$ missing in the $i$-th coordinate set of $\mathcal{I}_P$ is $\omega(i)$. As a result, we have $\pi(i) = \omega(i)$; otherwise, in the submatrix of $B$ whose columns are indexed by the $(i+1)$-st coordinate set of $\mathcal{I}_P$, the $\omega(i)$-th row would consist entirely of zeros, which, in turn, would contradict the fact that such a coordinate set represents a basis of the positroid $P$.
	
	The above argument also applies when $i = 2n$ provided that we extend the domain of $\omega$ to $[2n+1]$ and set $\omega(2n+1) = 0$.
	
	Note that $\pi(1) = n+1$ follows immediately from the minimality of the second coordinate set of $\mathcal{I}_P$ and the fact that $B_2, \dots, B_n, B_{n+1}$ are linearly independent.
	
	Now suppose that $1 < i \le n$ and $i-1 \notin \omega(J)$. The minimality of the coordinate sets of $\mathcal{I}_P$ implies that all the indices $i, \dots, n$ appear in the $i$-th coordinate set. Furthermore, Lemma~\ref{lem:removing a column of B does not affect its rank} implies that $i-1$ does not belong to the $i$-th coordinate set of $\mathcal{I}_P$. Since no $j \in J$ has weight $i-1$, the $(i-1)$-st and $i$-th rows of the maximal submatrix of $B$ determined by the column index set $\{n+1, \dots, 2n\}$ are equal. Consequently, we have $\pi(i) = i-1$; otherwise the associated maximal submatrix of $B$ determined by the indices of the $i$-th coordinate set of $\mathcal{I}_P$ would have the $i$-th and $(i+1)$-st rows identical, which would contradict the fact that the $i$-th coordinate set of $\mathcal{I}_P$ represents a basis of $P$.
	
	Finally, suppose that $1 < i \le n$ and $i-1 \in \omega(J)$. Since not two elements of $J$ have the same weight, there is at most one $j \in J$ such that $\omega(j) = i-1$. As before, all the indices $i, \dots, n+1$ appear in the $i$-th coordinate set of $\mathcal{I}_P$ (because $i>1$). Each column $B_k$, for $n < k \le 2n$ such that $\omega(k) = i-1$, is a linear combination of the columns $B_i, \dots, B_{n+1}$. Therefore such indices $k$ do not appear in the $i$-th coordinate set of $\mathcal{I}_P$. By Lemma~\ref{lem:removing a column of B does not affect its rank}, it follows that $i-1$ does not appear in the $i$-th coordinate set of $\mathcal{I}_P$. Thus, $\pi(i) = j$, where $j \in [2n]$ satisfies that $\omega(j) = i-1$; otherwise, in the submatrix of $B$ whose columns are indexed by the $(i+1)$-st coordinate set of $\mathcal{I}_P$, the $(i-1)$-st row would consist entirely of zeros, which would contradict that the $(i-1)$-st coordinate set of $\mathcal{I}_P$ represents a basis of $P$. By minimality of the $(i+1)$-st coordinate set of $\mathcal{I}_P$ one finds that $j \in J$.
\end{proof}

As the next theorem indicates, $\pi^{-1}$ is a $2n$-cycle satisfying a very special property.

\begin{theorem} \label{thm:main result}
	$\pi^{-1}$ is a $2n$-cycle $(1 \ j_1 \ \dots \ j_{2n-1})$ satisfying the next two conditions:
	\begin{enumerate}
		\item in the sequence  $(1, j_1, \dots, j_{2n-1})$ the elements $1,\dots,n$ appear in increasing order while the elements $n+1, \dots, 2n$ appear in decreasing order;
		\item for every $1 \le k \le 2n-1$, the set $\{1, j_1, \dots, j_k\}$ contains at least as many elements of the set $\{1,\dots,n\}$ as elements of the set $\{n+1, \dots, 2n\}$.
	\end{enumerate}
\end{theorem}

\begin{proof}
	From Proposition~\ref{prop:explicity function for decorated permutations} we immediately deduce that if $\pi(i) = j$ for $1 < i \le 2n$, then $\omega(i) = \omega(j)$ when $i > n$ and $\omega(i) = \omega(j) + 1$ when $i \le n$. This implies, in particular, that $\omega(i) \ge \omega(j)$. Suppose, by way of contradiction, that $\pi^{-1}$, and so $\pi$, is not a $2n$-cycle. Then there is a cycle $(i_1 \ i_2 \ \dots \ i_k)$ in the canonical cycle-type decomposition of $\pi$ that does not contain $1$. Therefore one has
	\[
		\omega(i_1) \ge \omega(i_2) \ge \dots \ge \omega(i_k) \ge \omega(i_1),
	\]
	which implies $\omega(i_1) = \omega(i_2) = \dots = \omega(i_k)$. Since $\{i_1, \dots, i_k\}$ does not contain $1$, it follows that $\{i_1, \dots, i_k\} \subseteq \{n+1, \dots, 2n\}$, which is a contradiction. Hence the cycle-type decomposition of $\pi^{-1}$ contains only one cycle, which has length $2n$.
	
	Since $\pi(1) = n+1$, one gets that $\pi = (1 \ n\!+\!1 \ i_1 \ i_2 \ \dots \ i_{2n-2})$, where $\{i_1, \dots, i_{2n-2}\}$ is precisely the set $[2n] \setminus \{1,n+1\}$. As
	\begin{equation*}
		\omega(i_1) \ge \omega(i_2) \ge \dots \ge \omega(i_{2n-2}),
	\end{equation*}
	and $\omega(i) = i$ for every $i \in [n]$, the elements of the set $\{2, \dots, n\}$ must appear in the cycle $(1 \ n\!+\!1 \ i_1 \ i_2 \ \dots \ i_{2n-2})$ in decreasing order. On the other hand, by Proposition~\ref{prop:explicity function for decorated permutations} the indices of equal columns of $B$ (but perhaps the first one) show in increasing order and consecutively in the sequence $(1, n+1,  i_1,  i_2, \dots, i_{2n-2})$. Also, as the weight map $\omega$ is strictly decreasing when restricted to $J$, the elements of the set $\{n+1, \dots, 2n\}$ must appear in increasing order in the cycle $(1 \ n\!+\!1 \ i_1 \ i_2 \ \dots \ i_{2n-2})$. Thus, condition~(1) holds.
	
	To show condition~(2), write $\pi = (n\!+\!1 \ i_1 \ i_2 \ \dots \ i_{2n-2} \ 1)$ and suppose, by way of contradiction, that there exists $m \in \{1, \dots, 2n-2\}$ such that
	\begin{equation} \label{eq:comparison of cardinality}
		\big|\big\{1 \le  j \le m \mid i_j \in \{2, \dots, n\}\big\}\big| - 1 > \big|\big\{1 \le j \le m \mid i_j \in \{n+1,\dots, 2n\}\big\}\big|.
	\end{equation}
	Let $m$ be the minimal such index. By the minimality of $m$, one obtains that $i_m \in \{2, \dots, n\}$. Let $k$ be the maximum index such that $m \le k$ and $i_j \in \{2,\dots,n\}$ for each $j = m, \dots, k$. Note that $k < 2n-2$ and $\pi(i_k) \in \{n+2,\dots,2n\}$. Since
	\[
		|\{j \le k \mid 2 \le i_j \le n \}| = |\{i_k, \dots, n\}|
	\]
	and
	\[
		|\{j \le k \mid n+2 \le i_j \le 2n\}| = |\{n+2, \dots, \pi(i_k) - 1\}|,
	\]
	it follows by \eqref{eq:comparison of cardinality} that $(n - i_k + 1) - 1 > (\pi(i_k) - 1) - (n+2) + 1 = \pi(i_k) - n -2$, which implies $2n - \pi(i_k) + 1> i_k - 1$. On the other hand, the fact that all the entries of $A$ below and on the main diagonal equal $1$ implies that $\omega(j) \ge 2n - j + 1$ for every $n+1 \le j \le 2n$. Since $1 < i_k \le n$, one finds that $i_k = \omega(i_k) = \omega(\pi(i_k)) + 1$. As $n+1 \le \pi(i_k) \le 2n$, we have
	\[
		i_k - 1 = \omega(\pi(i_k)) \ge 2n - \pi(i_k) + 1 > i_k - 1,
	\]
	which is a contradiction. Hence, writing $\pi^{-1} = (1 \ j_1 \ \dots \ j_{2n-1})$, we will obtain that for $k = 1, \dots, 2n-1$, the set $\{1, j_1, \dots, j_k\}$ contains at least as many elements of the set $[n]$ as elements of the set $\{n+1, \dots, 2n\}$, which is condition~(2).
\end{proof}

\section{A Direct Way to Read The Unit Interval Positroid} \label{sec:reading the UIP from the Dyck matrix}

Throughout this section, let $P$ be a canonically $n$-labeled unit interval order with antiadjacency matrix $A$. Also, let $\mathcal{I} = \{[q_i, q_i + 1] \mid 1 \le i \le n\}$ be a canonical interval representation of $P$ (i.e., $q_1 < \dots < q_n$); Proposition~\ref{prop:interval representations of a canonically labeled UIO} ensures the existence of such an interval representation. In this section we describe a way to obtain the decorated permutation associated to the unit interval positroid induced by $P$ directly from either $A$ or $\mathcal{I}$. Such a description will reveal that the function $\rho \circ \psi \circ \varphi \colon \mathcal{U}_n \to \mathcal{P}_n$ introduced in Section~\ref{sec:background} is a bijection (Theorem~\ref{thm:fundamental bijection}).

Recall that the north and east borders of the Young diagram formed by the nonzero entries of $A$ give a path of length $2n$ that we call the \emph{semiorder path} of $A$. Let $B = (I_n | A') = \psi(A)$, where $\psi$ is the map introduced in Lemma~\ref{lem:correspondence between totally nonnegative square and rectangular matrix}. We will also associate a second path to $A$. Let the \emph{inverted path} of $A$ be the path consisting of the south and east borders of the Young diagram formed by the nonzero entries of $A'$. Note that the inverted path of $A$ is just the reflection in a horizontal line of the semiorder path of $A$. Example~\ref{ex:decorated permutation from Dyck matrices} sheds light upon the statement of the next theorem, which describes a way to find the decorated permutation associated to the unit interval positroid induced by $P$ directly from $A$.

\begin{theorem} \label{thm:decorated permutations from the antiadjacency matrices of UIOs}
	If we number the $n$ vertical steps of the semiorder path of $A$ from bottom to top in increasing order with $\{1,\dots,n\}$ and the $n$ horizontal steps from left to right in increasing order with $\{n+1, \dots, 2n\}$, then by reading the semiorder path in the northwest direction, we obtain the decorated permutation associated to the unit interval positroid induced by $P$.
\end{theorem}

\begin{proof}
	Let $\pi^{-1}$ be the decorated permutation associated to the unit interval positroid induced by $P$. Label the $n$ vertical steps of the inverted path of $P$ from top to bottom in increasing order using the label set $[n]$, and we label the $n$ horizontal steps from left to right in increasing order using the label set $\{n+1, \dots, 2n\}$ (see Example~\ref{ex:decorated permutation from Dyck matrices}). Proving the theorem amounts to showing that we can obtain $\pi$ (the inverse of the decorated permutation) by reading the inverted path in the northeast direction. Let $(s_1, s_2, \dots, s_{2n})$ be the finite sequence obtained by reading the inverted path in the northeast direction. Since the first step of the inverted path is horizontal and the last step of the inverted path is vertical, $s_1 = n+1$ and $s_{2n} = 1$. Thus, it suffices to check that $\pi(s_k) = s_{k+1}$ for $k = 1, \dots, 2n-1$.
	
	Suppose first that the $k$-th step of the inverted path is horizontal, and so located right below the last nonzero entry of the $s_k$-th column of $B$. If the $(k+1)$-st step is also horizontal, then $s_{k+1} = s_k + 1$, which means that $\pi(s_k) = s_k + 1$ and so $\pi(s_k) = s_{k+1}$. On the other hand, if the $(k+1)$-st step is vertical, then $s_k = 2n$ or $s_k + 1$ is in the set of principal indices $J$ of $B$; in both cases, $\pi(s_k) = \omega(s_k)$, the number of vertical steps from the top to $s_k$, namely, $s_{k+1}$. Hence $\pi(s_k) = s_{k+1}$.
	
	Suppose now that the $k$-th step of the inverted path is vertical. Clearly, this implies that $1 \le s_k \le n$. If the $(k+1)$-st step is also vertical, then $s_{k+1} = s_k - 1$. Because steps $k$ and $k+1$ are both vertical, $A'$ does not contain any column with weight $s_k - 1$. As a result, $\pi(s_k) = s_k - 1 = s_{k+1}$. Finally, if the $(k+1)$-st step is horizontal, then $\{s_{k+1}\} = J \cap \omega^{-1}(s_k - 1)$ and, by Proposition~\ref{prop:explicity function for decorated permutations}, we find that $\pi(s_k) = s_{k+1}$.
\end{proof}

\begin{example} \label{ex:decorated permutation from Dyck matrices}
	In Figure~\ref{fig:Dyck matrix showing its semiorder path}, we can see displayed the antiadjacency matrix $A$ of the canonically $5$-labeled unit interval order $P$ introduced in Example~\ref{ex:main example} and the matrix $\psi(A)$ both showing their respective semiorder and inverted path encoding the decorated permutation $\pi = (1 \ 2 \ 1\!0 \ 3 \ 9 \ 4 \ 8 \ 7 \ 5 \ 6)$ associated to the positroid induced by $P$.
	\begin{figure}[h]
		\centering
		\includegraphics[width = 3.9cm]{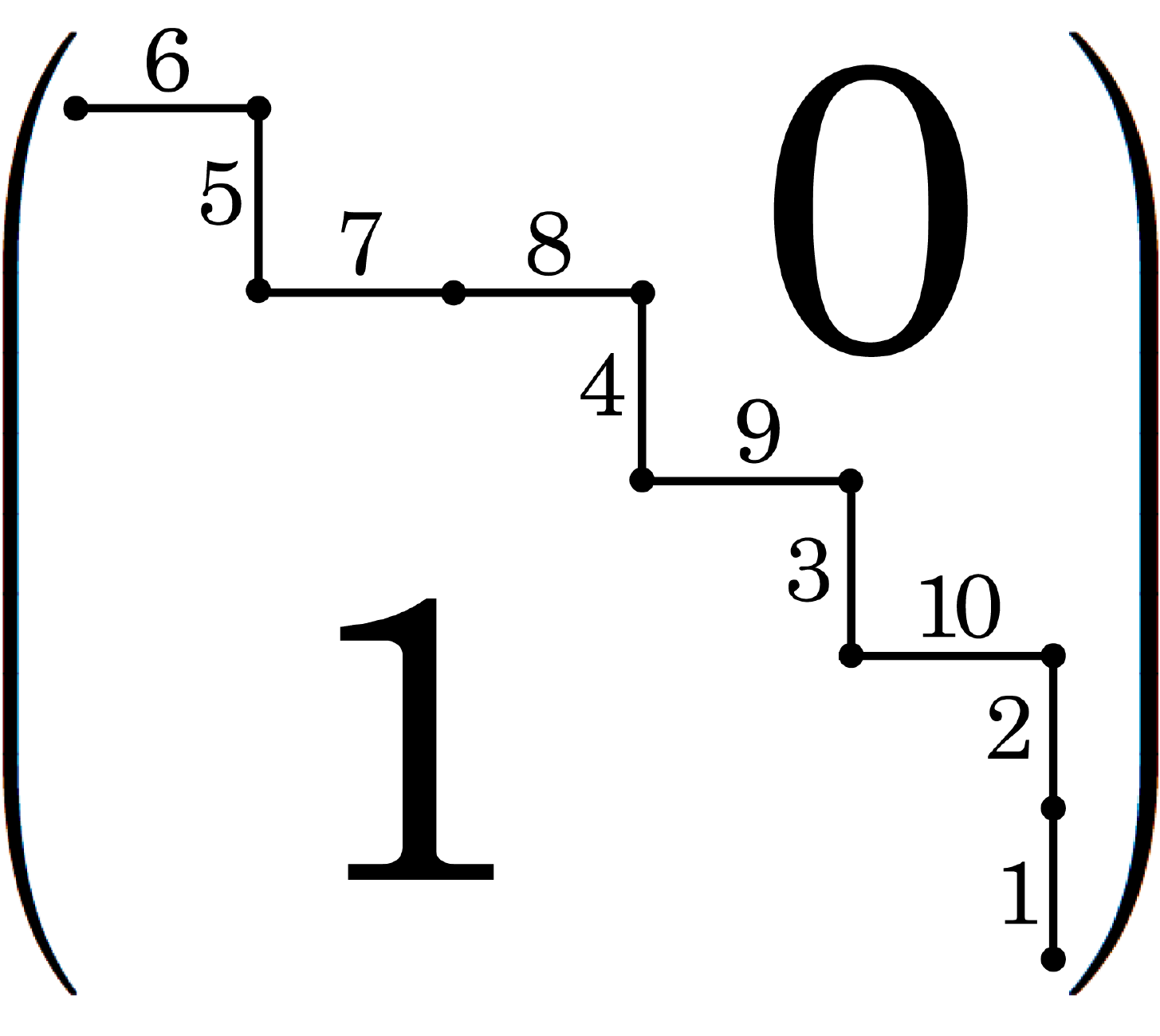}
		\raisebox{-0.13\height}{\includegraphics[width = 1.6cm]{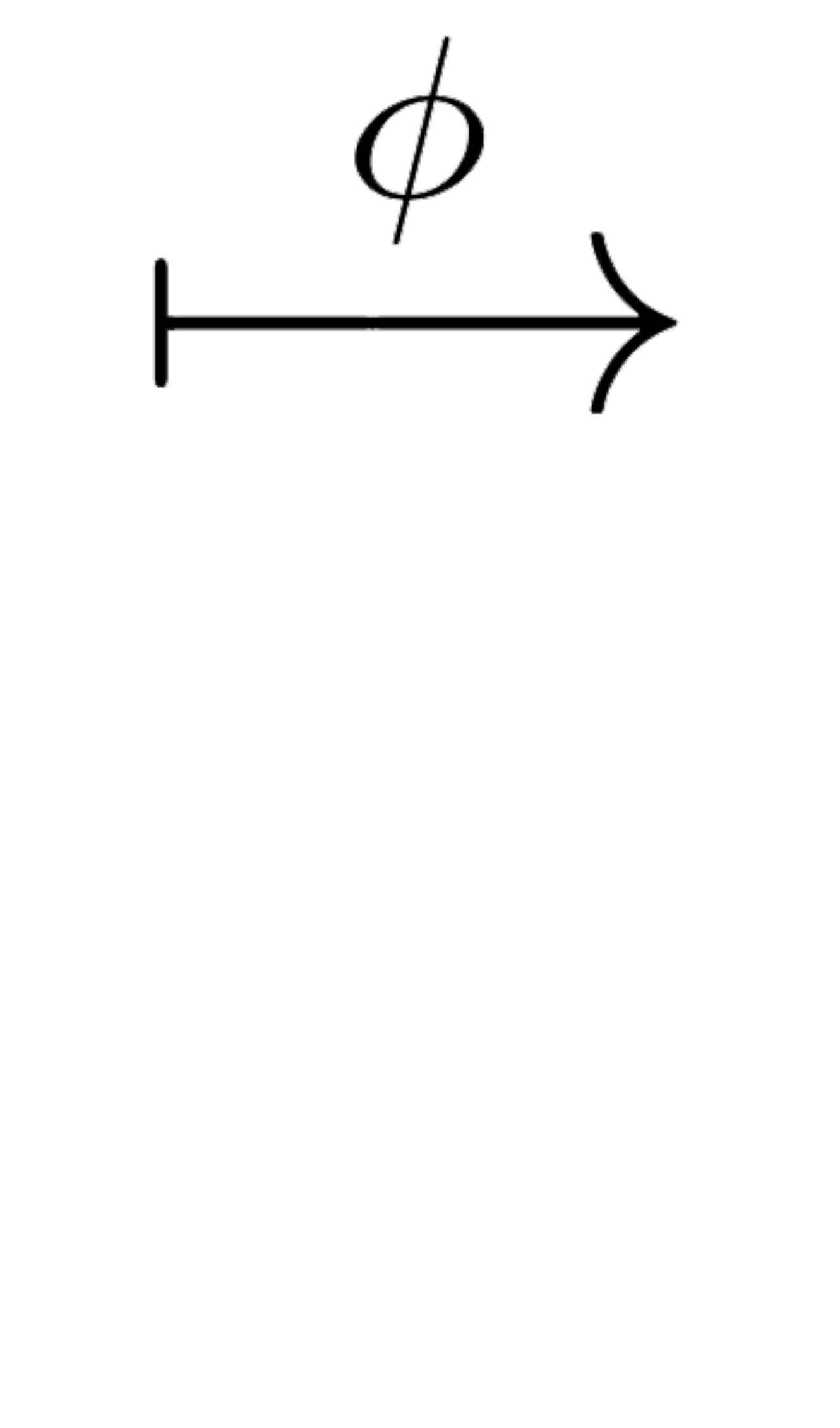}}
		\includegraphics[width = 7.4cm]{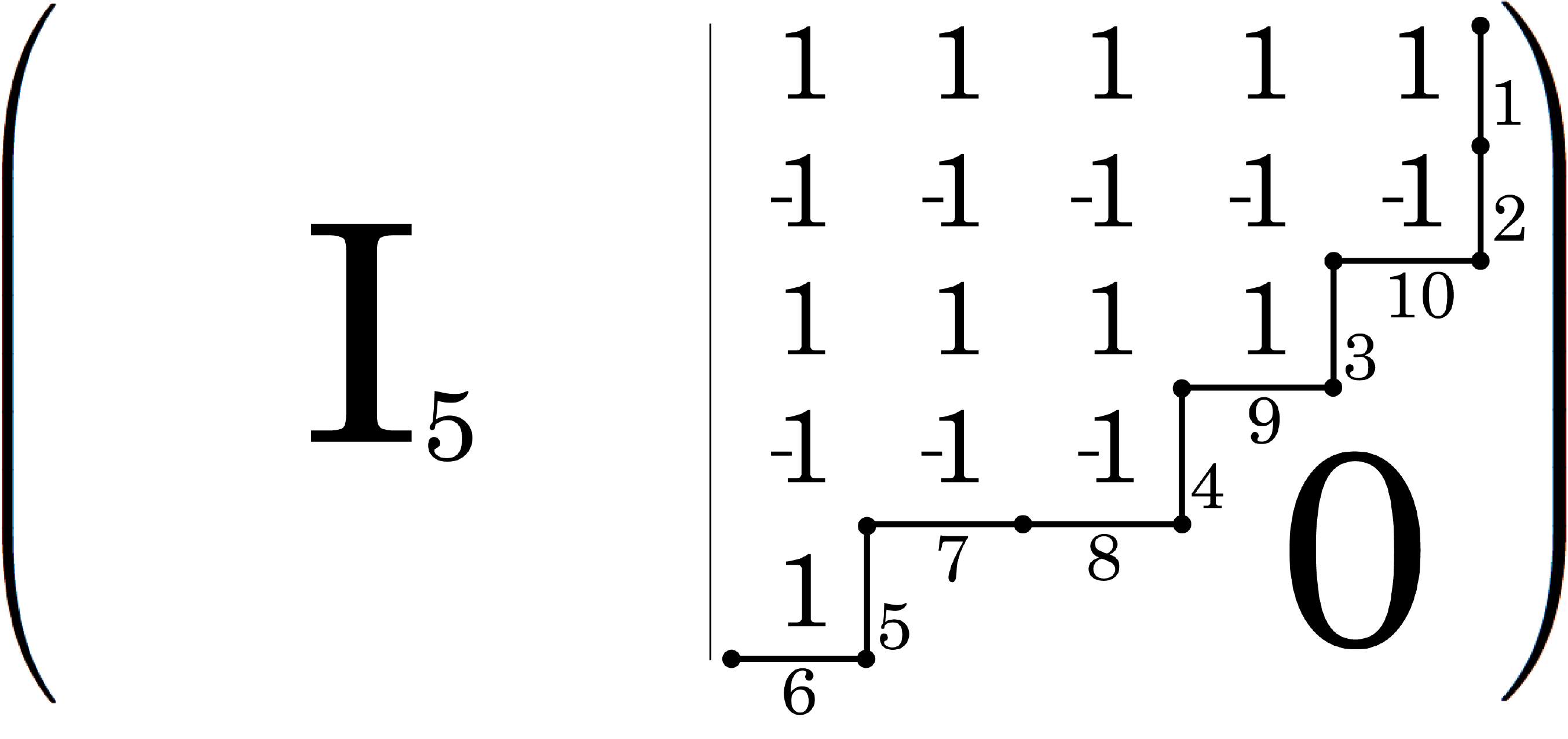}
		\caption{Dyck matrix $A$ and its image $\psi(A)$ exhibiting the decorated permutation $\pi$ along their semiorder path and inverted path, respectively.}
		\label{fig:Dyck matrix showing its semiorder path}
	\end{figure}
\end{example}

As a consequence of Theorem~\ref{thm:decorated permutations from the antiadjacency matrices of UIOs}, we can deduce that the map $\rho \circ \psi \circ \varphi \colon \mathcal{U}_n \to \mathcal{P}_n$, where $\rho$, $\psi$, and $\varphi$ are as defined in Section~\ref{sec:background} and Section~\ref{sec:canonically labelings on UIO}, is indeed a bijection.

\begin{lemma} \label{lem:Dyck paths specified by decorated permutations}
	The set of $2n$-cycles $(1 \ j_1 \ \dots \ j_{2n-1})$ satisfying conditions (1) and (2) of Theorem~\ref{thm:main result} is in bijection with the set of Dyck paths of length $2n$.
\end{lemma}

\begin{proof}
	We can assign a Dyck path $D$ of length $2n$ to the $2n$-cycle $(1\!\!=\!\!j_0 \ j_1 \ \dots \ j_{2n-1})$ by thinking of the entries $j_i \in \{1,\dots,n\}$ as ascending steps of the Dyck path $D$ and the entries $j_i \in \{n+1, \dots,2n\}$ as descending steps of $D$. The fact that such an assignment yields the desired bijection is straightforward.
\end{proof}
\medskip

\begin{theorem} \label{thm:fundamental bijection}
	The map $\rho \circ \psi \circ \varphi \colon \mathcal{U}_n \to \mathcal{P}_n$ is a bijection.
\end{theorem}

\begin{proof}
	By definition of $\mathcal{P}_n$, it follows that $\rho \circ \psi \circ \varphi$ is surjective. Since $|\mathcal{U}_n|$ is the $n$-th Catalan number, it suffices to show that $|\mathcal{P}_n| \ge \frac{1}{n+1} \binom{2n}{n}$. To see this, one can take a $2n$-cycle $\sigma = (1 \ j_1 \ \dots \ j_{2n-1})$ satisfying conditions (1) and (2) of Theorem~\ref{thm:main result}, and consider the Dyck path $D$ specified by $\sigma$ as in Lemma~\ref{lem:Dyck paths specified by decorated permutations}. By Theorem~\ref{thm:reading decorated permutation from Dyck matrix}, the Dyck matrix whose semiorder path is the reverse of $D$ induces a unit interval positroid with decorated permutation $\sigma$. Because the decorated permutation associated to a positroid is unique, Lemma~\ref{lem:Dyck paths specified by decorated permutations} guarantees that $|\mathcal{P}_n| \ge \frac{1}{n+1} \binom{2n}{n}$. Hence $\rho \circ \psi \circ \varphi$ is bijective.
\end{proof}

\begin{cor}
	The number of unit interval positroids on the ground set $[2n]$ equals the $n$-th Catalan number.
\end{cor}

We conclude this section by describing how to decode the decorated permutation associated to the unit interval positroid induced by $P$ directly from its canonical interval representation $\mathcal{I}$. Labeling the left and right endpoints of the intervals $[q_i, q_i + 1] \in \mathcal{I}$ by $-$ and $+$, respectively, we obtain a $2n$-tuple consisting of pluses and minuses by reading from the real line the labels of the endpoints of all such intervals. On the other hand, we can have another \emph{plus-minus} $2n$-tuple if we replace the horizontal and vertical steps of the semiorder path of $A$ by $-$ and $+$, respectively, and then read it in southeast direction as indicated in the following example.

\begin{example} \label{ex:plus-minus vectors of the semiorder and inverted paths}
	The figure below shows the antiadjacency matrix of the canonically $5$-labeled unit interval order $P$ from Example~\ref{ex:main example} and a canonical interval representation of $P$, both encoding the plus-minus $10$-tuple $(-,+,-,-,+,-,+,-,+,+)$, as described in the previous paragraph.
	\begin{figure}[h] \label{fig:plus-minus vectors}
		\centering
		\includegraphics[width = 3.8cm]{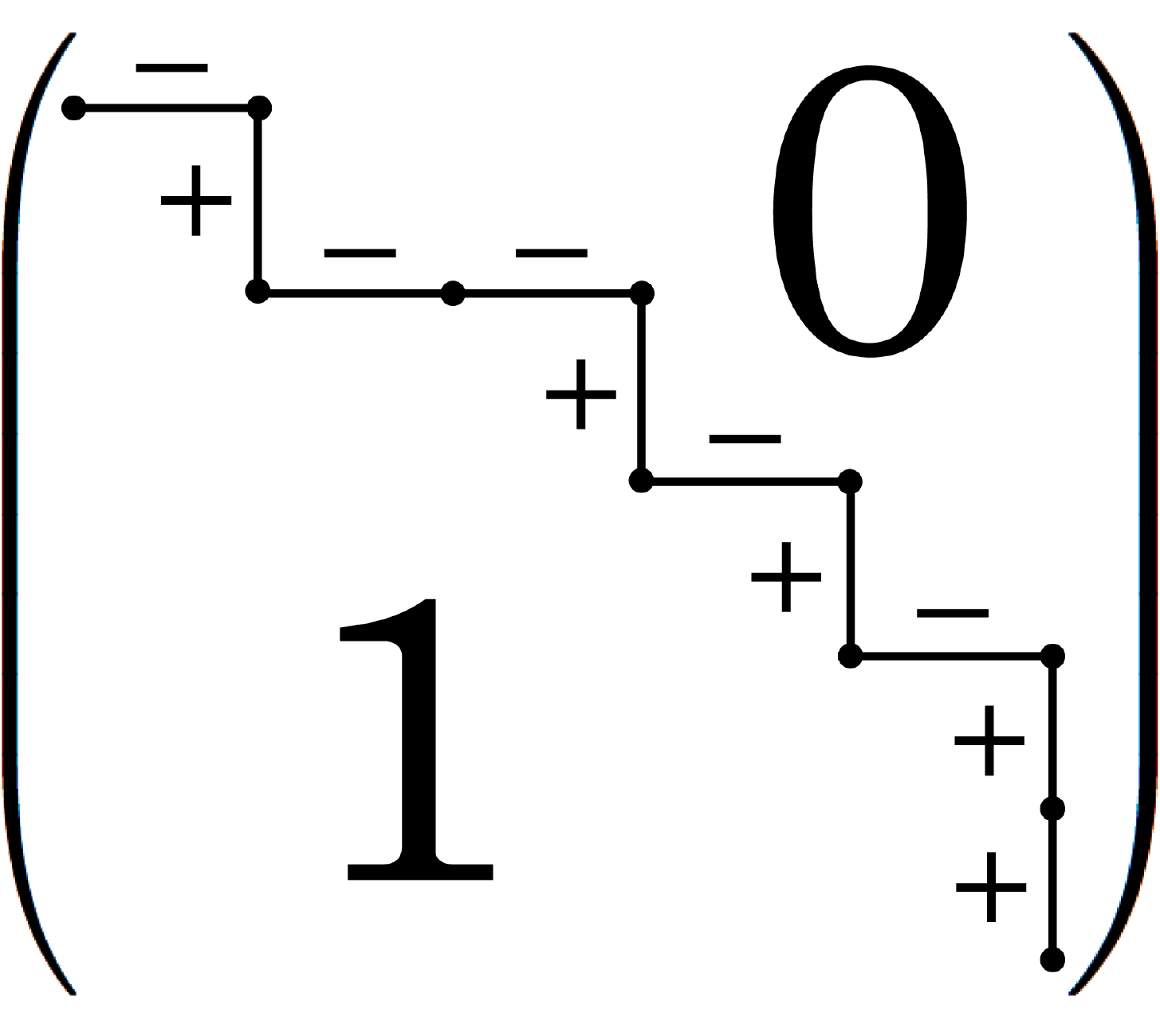} \
		\includegraphics[width = 1.2cm]{DoubleArrow} \
		\includegraphics[width = 7.5cm]{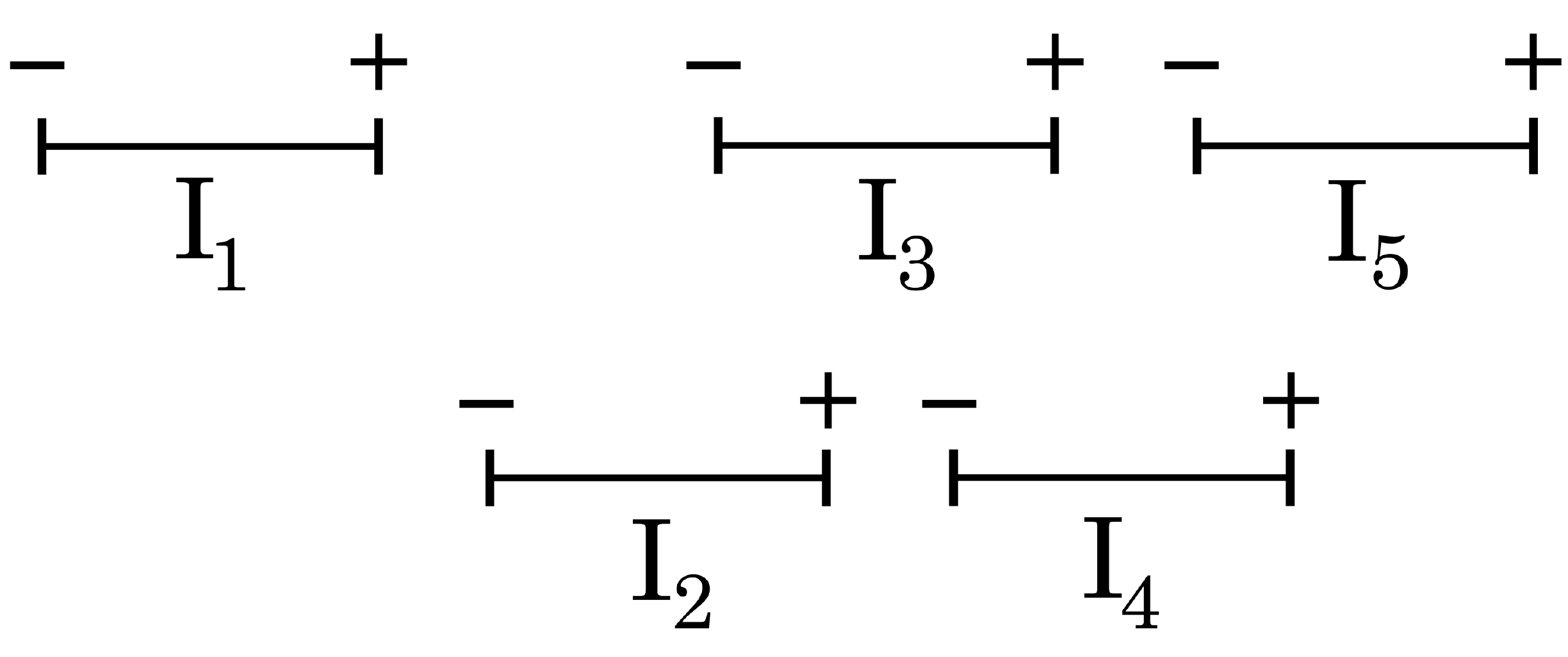}
		\caption{Dyck matrix and canonical interval representation of $P$ encoding the $10$-tuple $(-,+,-,-,+,-,+,-,+,+)$.}
		\label{fig:plus-minus vectors of the semiorder and inverted paths}
	\end{figure}
\end{example}

\begin{lemma} \label{lem:matrix and interval representations of UIO have the same pm-tuple}
	Let $\mathbf{a}_n =(a_1, \dots, a_{2n})$ and $\mathbf{b}_n = (b_1, \dots, b_{2n})$ be the $2n$-tuples with entries in $\{+,-\}$ obtained by labeling the steps of the semiorder path of $A$ and the endpoints of all intervals in $\mathcal{I}$, respectively, in the way described above. Then $\mathbf{a}_n = \mathbf{b}_n$.
\end{lemma}

\begin{proof}
	Let us proceed by induction on the cardinality $n$ of $P$. When $n=1$, both $\mathbf{a}_1$ and $\mathbf{b}_1$ are equal to $(-,+)$ and so $\mathbf{a}_1 = \mathbf{b}_1$. Suppose now that the statement of the lemma is true for every canonically $n$-labeled unit interval order, and assume that $P$ is a unit interval order canonically labeled by $[n+1]$ with antiadjacency matrix $A$ and canonical interval representation $\mathcal{I}$. Set $m = |\Lambda_{n+1}| - 1$. By Proposition~\ref{prop:a labeled UIO is canonical iff its antiadjacency matrix is Dyck}, the poset $P \! \setminus \! \{n+1\}$ is a unit interval order canonically labeled by $[n]$; therefore its associated plus-minus $2n$-tuples $\mathbf{a}'_n$ and $\mathbf{b}'_n$ are equal. Observe, in addition, that $\mathbf{b}_{n+1}$ can be recovered from $\mathbf{b}'_n$ by inserting the $-$ corresponding to the left endpoint of $q_{n+1}$ (labeled by $2n+2$) in the position $m+n+1$ (there are $n$ left interval endpoints and $m$ right interval endpoints to the left of $q_{m+1}$ in $\mathcal{I}$) and adding the $+$ corresponding to the right endpoint of $q_{n+1}$ (labeled by $1$) at the end. On the other hand, $\mathbf{a}_{n+1}$ can be recovered from $\mathbf{a}'_n$ by inserting the $-$ corresponding to the rightmost horizontal step of the semiorder path of $A$ in the position $m+n+1$ (there are $n$ horizontal steps and $m$ vertical steps before the last horizontal step of the semiorder path) and placing the $+$ corresponding to the vertical step labeled by $1$ in the last position. Hence $\mathbf{a}_{n+1} = \mathbf{b}_{n+1}$, and the lemma follows by induction. 
\end{proof}

As a consequence of Theorem~\ref{thm:decorated permutations from the antiadjacency matrices of UIOs} and Lemma~\ref{lem:matrix and interval representations of UIO have the same pm-tuple}, one obtains a way of reading the decorated permutation associated to the unit interval positroid induced by $P$ directly from $\mathcal{I}$.

\begin{cor} \label{cor:decorated permutations from interval representations of UIOs}
	Labeling the left and right endpoints of the intervals $[q_i, q_i + 1]$ by $n+i$ and $n+1-i$, respectively, we obtain the decorated permutation associated to the positroid induced by $P$ by reading these $2n$ labels from right to left on the real line.
\end{cor}

\begin{proof}
	By Lemma~\ref{lem:matrix and interval representations of UIO have the same pm-tuple}, the $2n$-tuple resulting from reading the set $\{1,\dots, 2n\}$ as indicated in Corollary~\ref{cor:decorated permutations from interval representations of UIOs} equals the $2n$-tuple resulting from reading the same set from the semiorder path of $A$ in northwest direction, as described in Theorem~\ref{thm:decorated permutations from the antiadjacency matrices of UIOs}. Hence the corollary follows immediately from Theorem~\ref{thm:decorated permutations from the antiadjacency matrices of UIOs}.
\end{proof}

\begin{example}
	The diagram below illustrates how to label the endpoints of a canonical interval representation of the $6$-labeled unit interval order $P$ shown in Figure~\ref{fig:UIO and its antiadjacency matrix} to obtain the decorated permutation $\pi = (1 \ 1\!2 \ 2 \ 3 \ 1\!1 \ 1\!0 \ 4 \ 5 \ 9 \ 6 \ 8 \ 7)$ associated to the positroid induced by $P$ by reading such labels from the real line (from right to left).
	\begin{figure}[h]
		\centering
		\includegraphics[width = 11cm]{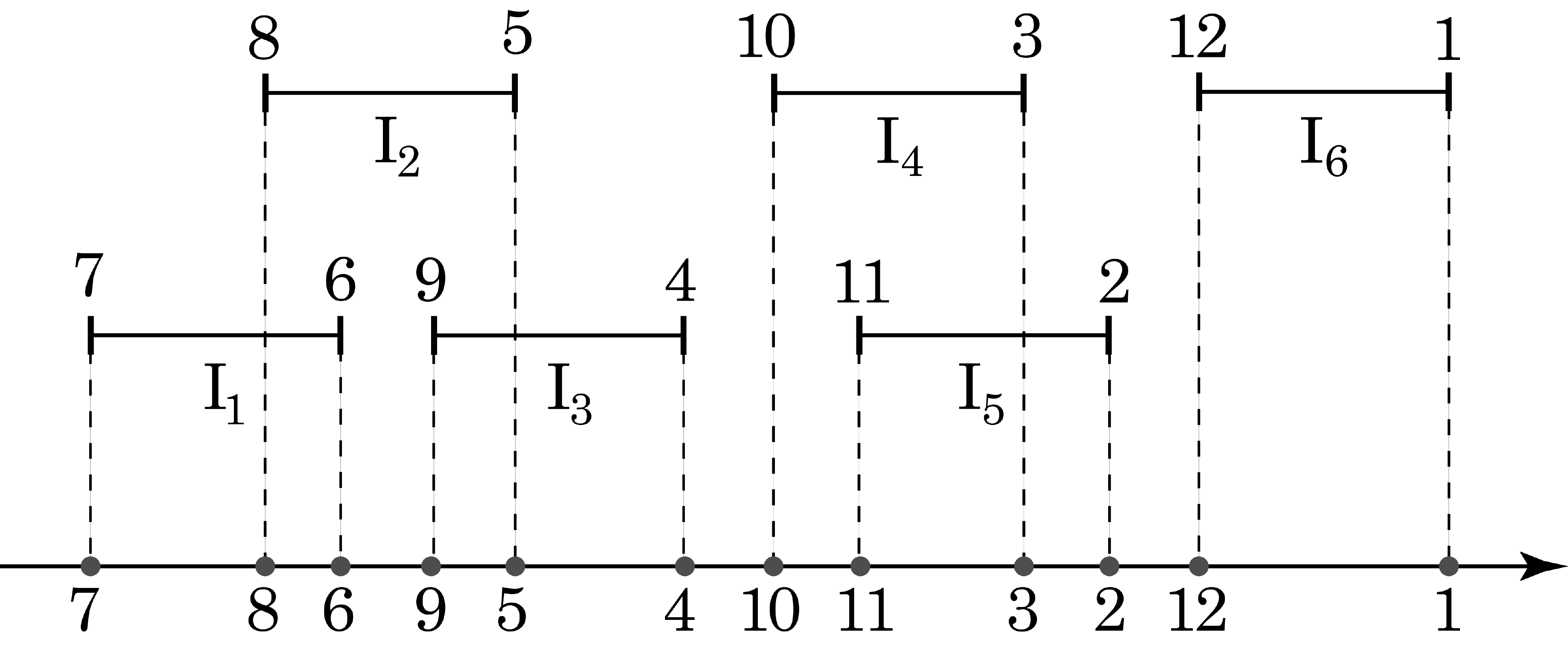}
		\caption{Decorated permutation $\pi$ encoded in a canonical interval representation of $P$.}
		\label{fig:decorated permutation from interval representation}
	\end{figure}
\end{example}

\section{The Le-diagram of a Unit Interval Positroid} \label{sec:Le-diagram}

The set consisting of all $d$-dimensional subspaces of $\mathbb{R}^n$, denoted by $\text{Gr}_{d,n}$, is called the \emph{real Grassmannian}. Elements in $\text{Gr}_{d,n}$ can also be understood as the orbits of the set of full-rank $d \times n$ real matrices under the left action of $\text{GL}_d(\mathbb{R})$. For $A \in \text{Mat}_{d,n}$ and $I \in \binom{[n]}{d}$, the \emph{Pl\"ucker coordinate} $\Delta_I(A)$ is the maximal minor of $A$ determined by the column set $I$. The embedding $\text{Gr}_{d,n} \hookrightarrow \mathbb{RP}^{\binom{n}{d} - 1}$ induced by the map $A \mapsto (\Delta_I(A))$ makes $\text{Gr}_{d,n}$ a projective variety. Let $\text{GL}^+_d(\mathbb{R})$ denote the set of real $d \times d$ matrices of positive determinant, and recall that $\text{Mat}^{\ge 0}_{d,n}$ is the set of real $d \times n$ matrices of rank $d$ having nonnegative maximal minors.

\begin{definition}
	The \emph{totally nonnegative Grassmannian}, denoted by $\text{Gr}^+_{d,n}$, is the set of orbits of $\text{Mat}^{\ge 0}_{d,n}$ under the left action of $\text{GL}^+_d(\mathbb{R})$, i.e., $\text{Gr}^+_{d,n} = \text{GL}^+_d(\mathbb{R}) \! \setminus \! \text{Mat}^{\ge 0}_{d,n}$.
\end{definition}

For a full-rank $d\times n$ real matrix $A$, let $M(A)$ denote the matroid represented by $A$, and let $[A]$ denote the element of $\text{Gr}_{d,n}$ represented by $A$. The \emph{matroid stratification} or \emph{Gelfand-Serganova stratification} of $\text{Gr}_{d,n}$ is the collection of all \emph{strata}
\[
	S_\mathcal{M} := \{[A] \in \text{Gr}_{d,n} \mid M(A) = \mathcal{M}\},
\]
where $\mathcal{M}$ runs over the set of rank $k$ representable matroids on the ground set $[n]$. For each stratum $S_{\mathcal{M}}$, we define a \emph{positroid cell} in $\text{Gr}^+_{d,n}$ by
\[
	S^+_\mathcal{M} = S_\mathcal{M} \cap \text{Gr}^+_{d,n}.
\]
Note that a representable matroid $\mathcal{M}$ is a positroid precisely when $S^+_\mathcal{M}$ is nonempty. The collection of nonempty positroid cells is called the \emph{cellular decomposition} of $\text{Gr}^+_{d,n}$. For further details, see \cite[Sections~2 and 3]{aP06}. 

Positroids, and therefore positroid cells, can be parameterized by a family of combinatorial objects called Le-diagrams. In this section, we characterize the Le-diagrams corresponding to unit interval positroids.

\begin{definition}
	Let $d,n \in \mathbb{N}$ and let $Y_\lambda$ be the Young diagram associated to a given partition $\lambda$. A \reflectbox{L}-\emph{diagram} (or \emph{Le-diagram}) $L$ of shape $\lambda$ and type $(d,n)$ is a Young diagram $Y_\lambda$ contained in a $d \times (n-d)$ rectangle, whose boxes are filled with zeros and pluses such that there is no zero entry which has simultaneously a plus entry above it in the same column and a plus entry to its left in the same row.
\end{definition}

The left picture in Figure~\ref{fig:Le-diagram of the unit interval positroid in Figure 6 and the corresponding diagram illustrating Lemma} shows a \reflectbox{L}-diagram of square shape and type $(6,12)$. As mentioned before, \reflectbox{L}-diagrams of type $(d,n)$ parameterize rank $d$ positroids on the ground set $[n]$. The next lemma yields a method to find the decorated permutation of a positroid given its corresponding \reflectbox{L}-diagram.

\begin{lemma} \cite[Lemma~4.8]{ARW16}\label{lem:decorated permutation from Le-diagram}
	The following algorithm is a bijection between \emph{\reflectbox{L}}-diagrams $D$ of
	type $(d,n)$ and decorated permutations $\pi$ on $n$ letters with $d$ weak excedances.
	\begin{enumerate}
		
		\item Replace each $+$ in the \emph{\reflectbox{L}}-diagram $D$ with an elbow joint $\,\raisebox{-0.1 \height}{\includegraphics[width=0.5cm]{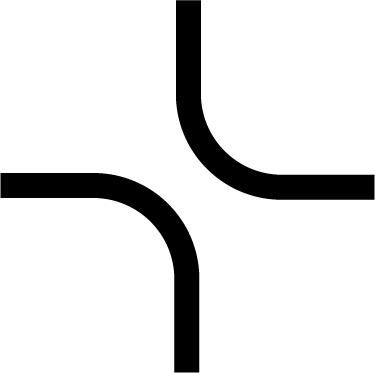}}$, and each $0$ in $D$ with a cross $\,\raisebox{-0.1 \height}{\includegraphics[width=0.5cm]{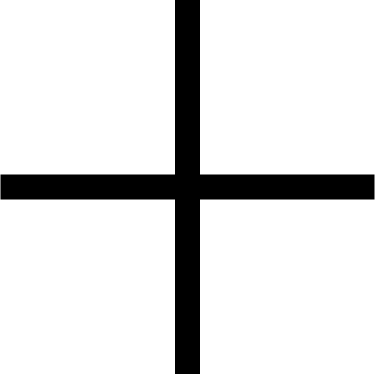}}$.
		
		\item Note that the south and east border of $Y_\lambda$ gives rise to a length-$n$ path from the northeast corner to the southwest corner of the $d \times (n-d)$ rectangle. Label the edges of this path with the numbers $1$ through $n$.
		
		\item Now label the edges of the north and west border of $Y_\lambda$ so that opposite horizontal edges and opposite vertical edges have the same label.
		
		\item View the resulting ``pipe dream" as a permutation $\pi \in S_n$ by following the ``pipes" from the northwest border to the southeast border of the Young diagram. If the pipe originating at label $i$ ends at the label $j$, we define $\pi(i) = j$.
		
		\item If $\pi(j) = j$ and $j$ labels two horizontal (respectively, vertical) edges of $Y_\lambda$, then we set $\pi(j) := \underline{j}$ (respectively, $\pi(j) := \overline{j}$).
	\end{enumerate}
\end{lemma}

The picture on the left of Figure~\ref{fig:Le-diagram of the unit interval positroid in Figure 6 and the corresponding diagram illustrating Lemma} shows the \reflectbox{L}-diagram corresponding to the positroid induced by the unit interval order displayed in Figure~\ref{fig:UIO and its antiadjacency matrix}. The ``pipe dream" of this \reflectbox{L}-diagram, as described in Lemma~\ref{lem:decorated permutation from Le-diagram}, is depicted on the right.

\begin{figure}[h]
	\centering
	\raisebox{0.1 \height}{\includegraphics[width = 4.3cm]{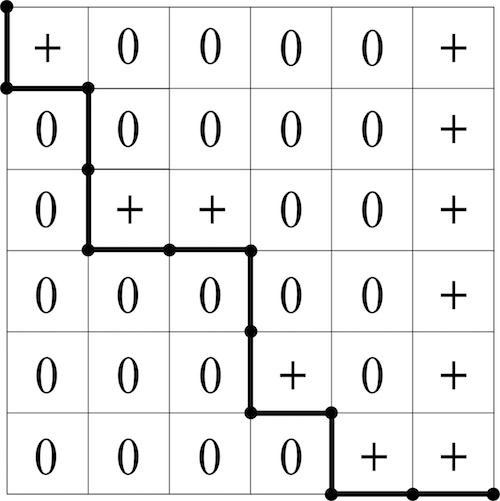}} \hspace{2cm}
	\includegraphics[width = 4.64cm]{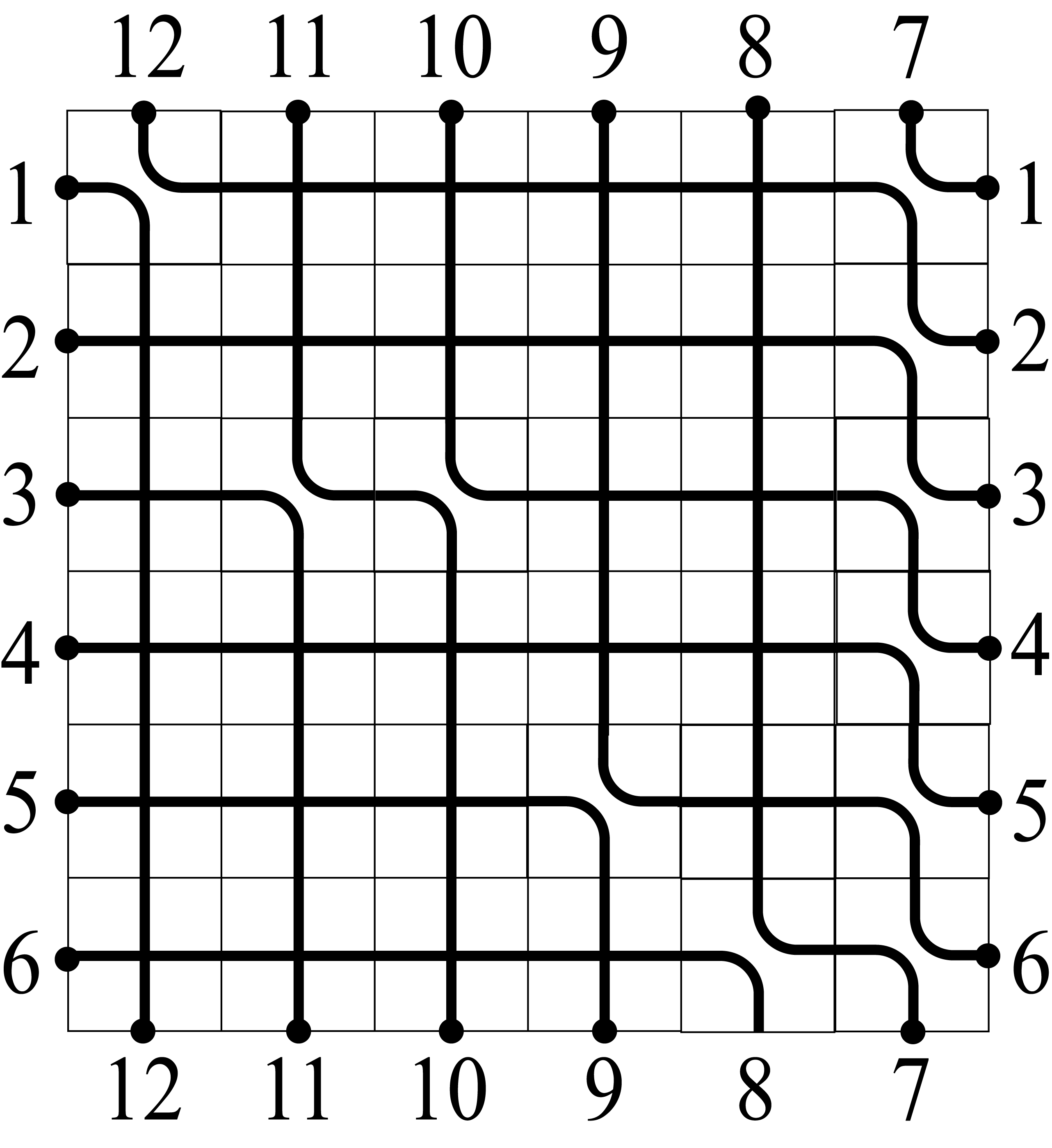}
	\caption{A Le-diagram and its corresponding ``pipe dream" as described in Lemma~\ref{lem:decorated permutation from Le-diagram}.}
	\label{fig:Le-diagram of the unit interval positroid in Figure 6 and the corresponding diagram illustrating Lemma}
\end{figure}

The next theorem provides a characterization of the \reflectbox{L}-diagrams corresponding to unit interval positroids.

\begin{theorem} \label{prop:description of the Le-diagram of a unit interval positroid}
	A \emph{\reflectbox{L}}-diagram $L$ of type $(n,2n)$ parameterizes a unit interval positroid on $[2n]$ if and only if its shape $\lambda$ is a square of size $n$ and $L$ satisfies the following two conditions:
	\begin{enumerate}
		\item every column has exactly one plus except the last one that has $n$ pluses;
		\vspace{3pt}
		\item the horizontal unit steps right below the bottom-most pluses are the horizontal steps of a length $2n$ Dyck path supported on the main diagonal of $L$.
	\end{enumerate}
\end{theorem}

\begin{proof}
	Suppose first that $L$ satisfies (1) and (2). To verify that $L$ corresponds to a unit interval positroid, let us use Lemma~\ref{lem:decorated permutation from Le-diagram} to compute its decorated permutation $\pi$ and show that $\pi^{-1}$ satisfies Proposition \ref{prop:explicity function for decorated permutations}. Note that $\pi^{-1}(1) = n+1$. For $i \in [2n] \! \setminus \! \{1\}$, we find $\pi^{-1}(i)$.
	
	Assume first that $i \in \{2, \dots, n\}$. If there is only one plus in the $(i-1)$-st row of $L$ (which means that $\omega(j) \neq i-1$ for each $j \in J$), it follows by Lemma~\ref{lem:decorated permutation from Le-diagram} that $\pi^{-1}(i) = i-1$. On the other hand (which means that there is exactly one principal element $j$ in $\omega^{-1}(i-1)$), one obtains that $\pi^{-1}(i)$ is the label of the first column (from right to left) of $L$ having a plus in the $(i-1)$-st row (which means $\pi^{-1}(i) = j$).
	
	Assume now that $i \in \{n+1, \dots, 2n\}$. If the bottom-most plus in the column of $L$ labeled by $i$ is the last plus from right to left in its row, which is labeled by $\omega(i)$, then by Lemma~\ref{lem:decorated permutation from Le-diagram} it follows that $\pi^{-1}(i) = \omega(i)$ (note, in this case, that $i = 2n$ or $i+1$ is a principal index). On the other hand, the columns of $L$ labeled by $i$ and $i+1$ are identical (i.e., $i+1$ is not a principal index), and Lemma~\ref{lem:decorated permutation from Le-diagram} yields $\pi^{-1}(i) = i+1$.
	
	Thus, $\pi^{-1}$ is as described in Proposition \ref{prop:explicity function for decorated permutations}, and so $\pi$ is the decorated permutation of a unit interval positroid on the ground set $[2n]$. As the number of \reflectbox{L}-diagrams satisfying the conditions above and the number of decorated permutations corresponding to unit interval positroids on the ground set $[2n]$ are equal to the $n$-th Catalan number, the proof follows.
\end{proof}

As a result of Theorem~\ref{prop:description of the Le-diagram of a unit interval positroid}, each unit interval positroid cell in $\text{Gr}^+_{k,n}$ can be indexed by a $\reflectbox{L}$-diagrams described in the same theorem. Postnikov proved that the positroid cell indexed by a $\reflectbox{L}$-diagram $L$ has dimension equal to the number of pluses of $L$ \cite[Theorem 4.6]{aP06}. This immediately implies the following corollary.

\begin{cor}
	The positroid cell of a unit interval positroid on the ground set $[2n]$ inside the cell decomposition of $\emph{Gr}_{n,2n}$ has dimension $2n-1$.
\end{cor}

\section{Adjacency of Unit Interval Positroid Cells} \label{sec:adjacency of UIP cells}

Given a decorated permutation $\pi$ on $n$ letters, its \emph{chord diagram} is constructed in the following way. First, place $n$ points labeled by $[n]$ in clockwise order around a circle. For all $i, j \in [n]$ with $i \neq j$ and $\pi(i) = j$, draw a directed chord from $i$ to $j$. If $\pi$ fixes $i$, then draw a directed chord from $i$ to $i$, oriented counterclockwise if and only if $\pi(i) = \overline{i}$. For $i, j \in [n]$, let $\text{Arc}(i,j)$ denote the set of points in the boundary circle of the chord diagram from $i$ to $j$ (both included) in clockwise order. Figure~\ref{fig:chord diagram of the running decorated permutation} shows an example of a chord diagram.
\begin{figure}[h]
	\centering
	\includegraphics[width = 4.0cm]{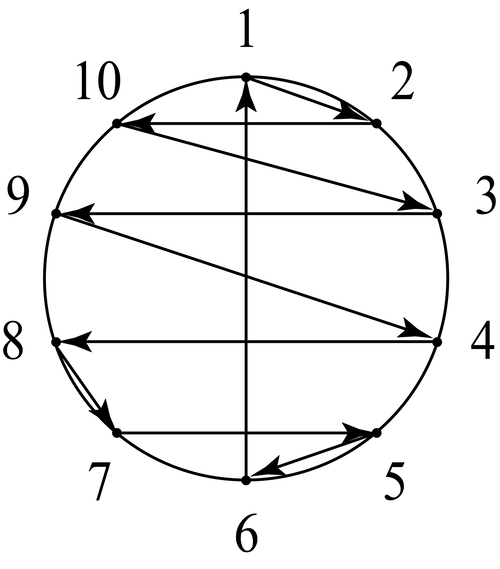}
	\caption{Chord diagram of the decorated permutation in Example~\ref{ex:main example}.}
	\label{fig:chord diagram of the running decorated permutation}
\end{figure}

Let $AD$ and $CB$ be two chords in the chord diagram of a decorated permutation $\pi$. We say that $AD$ and $CB$ form a \emph{crossing} if they intersect inside the circle or on its boundary, and this crossing is \emph{simple} if there are no other chords from $\text{Arc}(C, A)$ to $\text{Arc}(B, D)$. The left diagram in Figure~\ref{fig:crossings and alignments} shows a simple crossing. On the other hand, two chords $AB$ and $CD$ form an \emph{alignment} if they do not intersect and have a parallel orientation as shown in the right diagram of Figure~\ref{fig:crossings and alignments}. Notice that if $A$ and $B$ coincided in the right diagram below, then in order for $AB$ and $CD$ to have parallel orientation $AB$ must be a loop oriented counterclockwise. An alignment, as shown in the right side of the picture below, is said to be \emph{simple} if there are no other chords from $\text{Arc}(C,A)$ to $\text{Arc}(B,D)$.

\vspace{0.5cm}

\begin{figure}[h]
	\centering
	\includegraphics[width = 3.2cm]{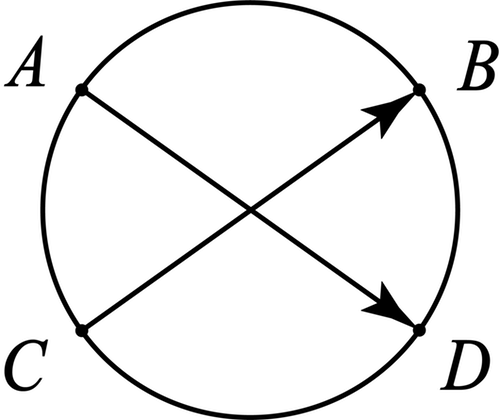} \hspace{2.5cm}
	\includegraphics[width = 3.2cm]{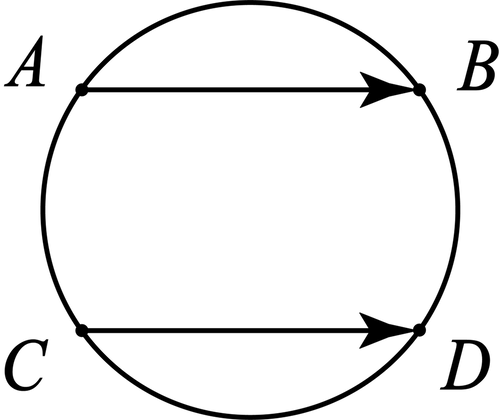}
	\caption{A simple crossing on the left and a simple alignment on the right.}
	\label{fig:crossings and alignments}
\end{figure}

\vspace{0.5cm}

Let $\pi_1$ and $\pi_2$ be two decorated permutations of the same size $n$. We say that $\pi_1$ \emph{covers} $\pi_2$, and write $\pi_1 \to \pi_2$, if the chord diagram of $\pi_2$ is obtained by turning a simple crossing of $\pi_1$ into a simple alignment. This is depicted in Figure~\ref{fig:covering relation}.

\vspace{0.5cm}

\begin{figure}[h]
	\centering
	\includegraphics[width = 3cm]{Crossing.png} \hspace{1cm}
	\raisebox{2.5 \height}{\includegraphics[width = 2cm]{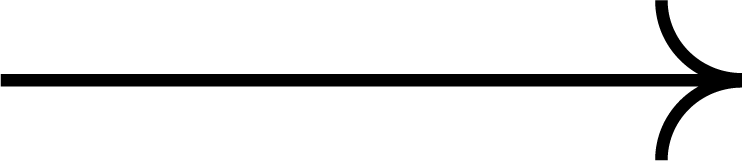}} \hspace{1cm}
	\includegraphics[width = 3cm]{Alignment.png}
	\caption{A covering relation.}
	\label{fig:covering relation}
\end{figure}

\vspace{0.5cm}

If the points $A$ and $B$ happen to coincide, then the chord from $A$ to $B$ in the chord diagram of $\pi_2$ degenerates to a counterclockwise loop. Similarly, if the points $C$ and $D$ coincide, then the chord from $C$ to $D$ in the chord diagram of $\pi_2$ becomes a clockwise loop. Finally, if $A=B$ or $C=D$, then the loops at $A$ and $C$ in the chord diagram of $\pi_2$ must be counterclockwise and clockwise, respectively. These three types of covering relations, illustrated in Figure~\ref{fig:degenerate covering relations}, are said to be \emph{degenerate}.
\begin{figure}[h]
	\centering
	\includegraphics[width = 3cm]{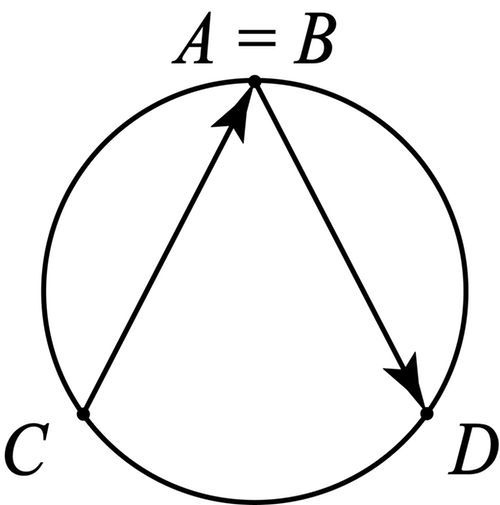} \hspace{1cm}
	\raisebox{2.7 \height}{\includegraphics[width = 2cm]{ArrowCovering.png}} \hspace{1cm}
	\includegraphics[width = 3cm]{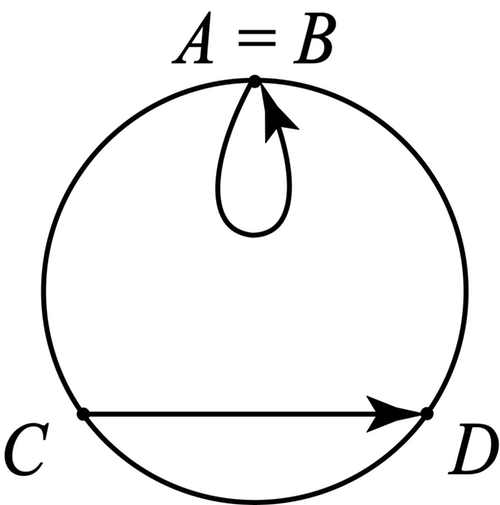} \vspace{1.4cm} %\vspace{1.5cm}
	
	\includegraphics[width = 3cm]{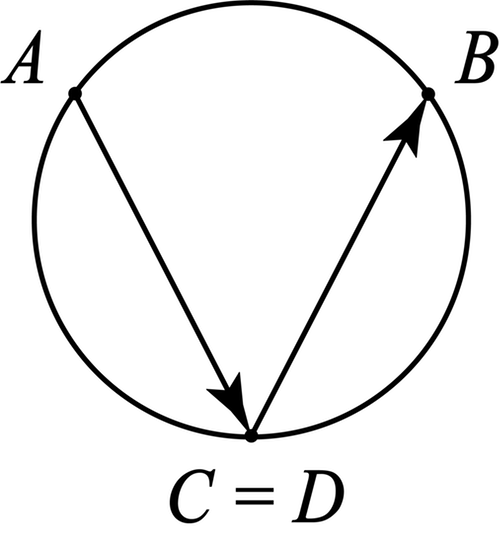} \hspace{1cm}
	\raisebox{3.4 \height}{\includegraphics[width = 2cm]{ArrowCovering.png}} \hspace{1cm}
	\includegraphics[width = 3cm]{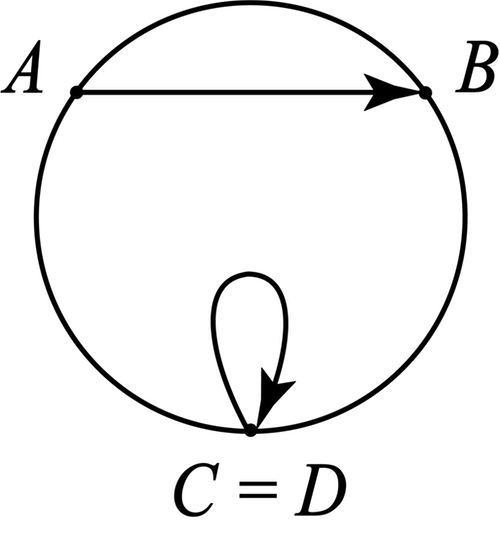} \vspace{0.5cm} %\vspace{1cm}
	
	\includegraphics[width = 2.7cm]{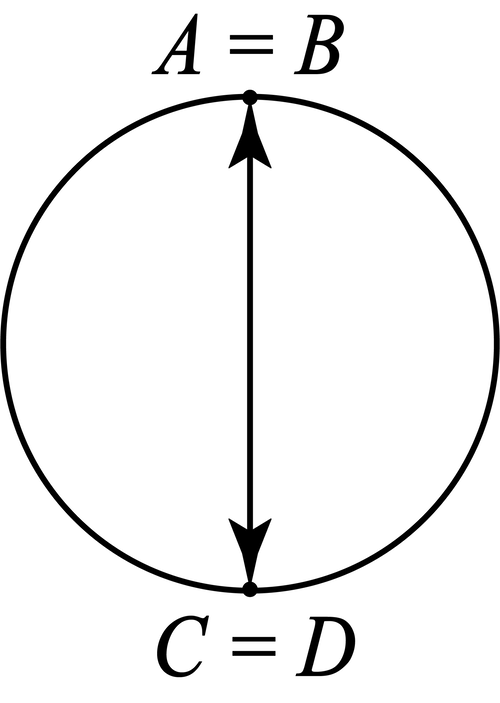} \hspace{1.15cm}
	\raisebox{3.6 \height}{\includegraphics[width = 2cm]{ArrowCovering.png}} \hspace{1.15cm}
	\includegraphics[width = 2.7cm]{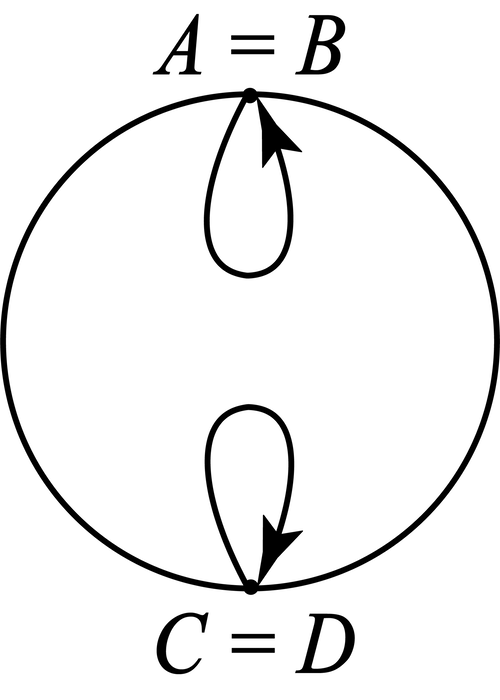}
	
	\caption{The three degenerate covering relations.}
	\label{fig:degenerate covering relations}
\end{figure}

Two positroid cells are \emph{adjacent} if the decorated permutation parameterizing them cover a common decorated permutation. Here is a necessary and sufficient condition for two unit interval positroid cells to be adjacent.

\begin{prop} \label{prop:condition for adjacency}
	Let $P_1$ and $P_2$ be two distinct rank $n$ unit interval positroids and $\pi_1$ and $\pi_2$ their respective decorated permutations. Then $P_1$ and $P_2$ label adjacent positroid cells if and only if there exists $i \in [2n] \! \setminus \! \{1, n+1\}$ such that when $i$ is removed from the cycle decomposition of $\pi_1$ and $\pi_2$ the resulting cycles are equal.
\end{prop}

\begin{proof}
	Let $C_1$ and $C_2$ be the chord diagrams of $\pi_1$ and $\pi_2$, respectively. Assume first that $P_1$ and $P_2$ label adjacent positroid cells whose decorated permutations both cover a permutation $\pi$. Let $C$ denote the chord diagram of $\pi$. Theorem~\ref{thm:main result} ensures that $C_1$ and $C_2$ have a directed edge from $n+1$ to $1$ and their non-degenerate simple crossings occur only along this edge. Unlike non-degenerate coverings, degenerate coverings increase the number of fixed points; therefore $\pi_1 \to \pi$ is a degenerate covering relation if and only if so is $\pi_2 \to \pi$. If $\pi_1 \to \pi$ and $\pi_2 \to \pi$ were both non-degenerate coverings, then the fact that both covering relations uncross the chord from $n+1$ to $1$ would imply that both $\pi_1$ and $\pi_2$ can be uniquely recovered from $\pi$, as the other chord being uncrossed in both covering relations must be the chord from $\pi^{-1}(1)$ to $\pi(n+1)$. This, in turn, would contradict that $\pi_1 \neq \pi_2$. As a result, both $\pi_1 \to \pi$ and $\pi_2 \to \pi$ are degenerate coverings. As $\pi_1$ and $\pi_2$ are $2n$-cycles, $\pi$ fixes exactly one element $i \in [2n] \! \setminus \! \{1, n+1\}$. Moreover, $\pi$ is the result of removing $i$ from the cycle decomposition of any of the permutations $\pi_1$ or $\pi_2$.
	
	Conversely, suppose that for some $i \in [2n] \! \setminus \! \{1,n+1\}$, removing $i$ from the cycle decomposition of either $\pi_1$ or $\pi_2$ produces the same $(2n-1)$-cycle $\pi$. In this case, $\pi_1 \to \pi$ and $\pi_2 \to \pi$ are degenerate covering relations. Hence $\pi_1$ and $\pi_2$ are adjacent and the proof follows.
\end{proof}

\begin{example}
	There are a total of five unit interval positroids on the ground set $[6]$. Let $\pi_1, \dots, \pi_5$ be their five corresponding decorated permutations. These permutations are illustrated in the top row of Figure~\ref{fig:adjacency of unit interval positroids on [6].} via their chord diagrams. The bottom row of the same figure shows the chord diagrams of four of the decorated permutations covered by the $\pi_i$'s. Although there are more than four decorated permutations covered by the $\pi_i$'s, those depicted below are enough to obtain all possible adjacency relations between the positroid cells parameterized by the $\pi_i$'s. The exterior long arrows below represent covering relations.  
	\begin{figure}[h]
		\centering
		\includegraphics[width = 15cm]{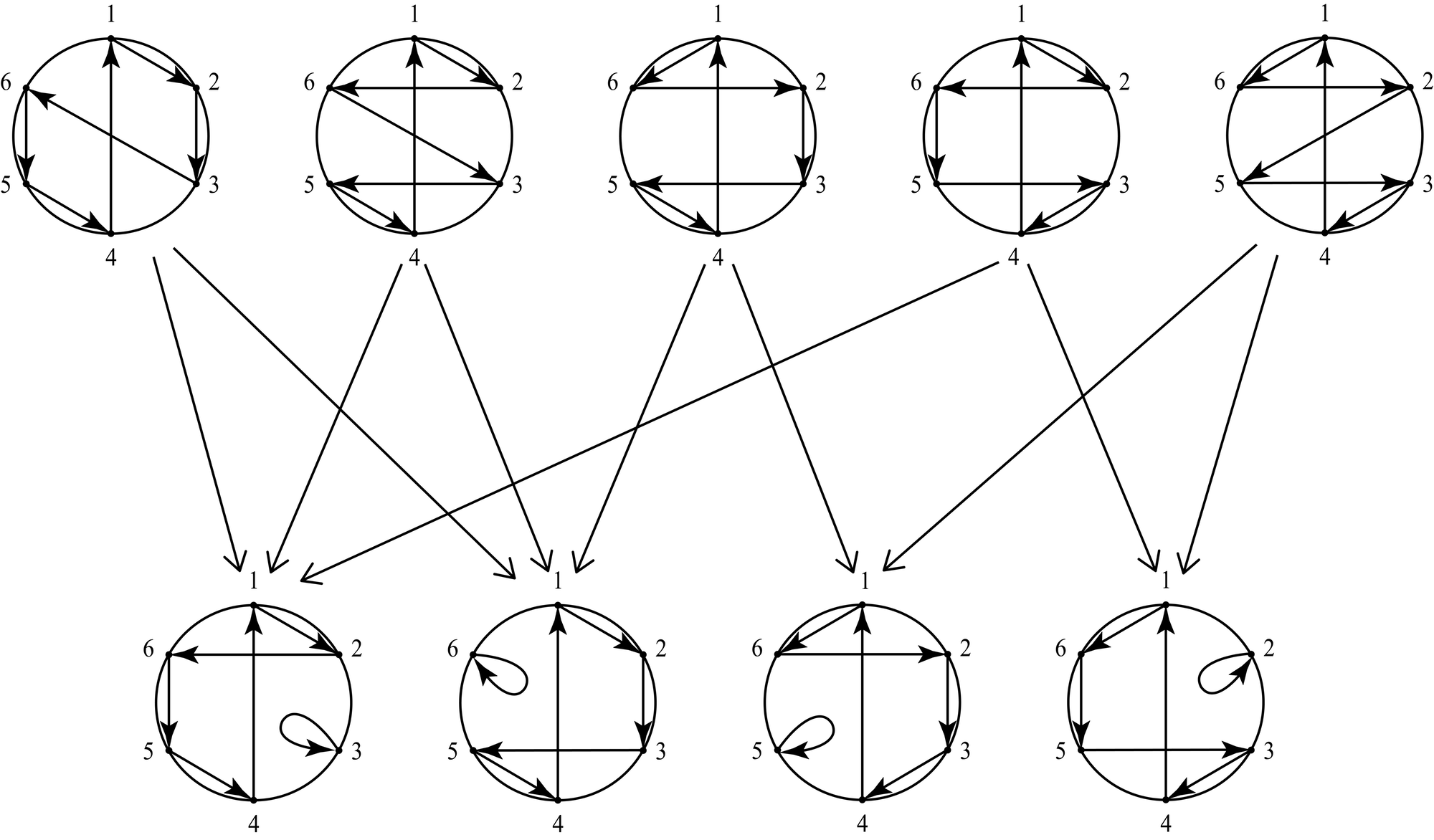}
		\caption{Subposet of $\text{CB}_{3,6}$ illustrating the adjacency relations among the unit interval positroid cells of dimension $5$.}
		\label{fig:adjacency of unit interval positroids on [6].}
	\end{figure}
\end{example}

It was proved in \cite{aP06} that if $\pi_1$ and $\pi_2$ are two decorated permutations such that $\pi_1 \to \pi_2$, then they both have the same number of weak excedances. Thus, the set of all decorated permutations of $[2n]$ having $n$ excedances can be regarded as a poset with order given by the transitive closure of the covering relation ``$\to$"; this poset is called the \emph{cyclic Bruhat order} and is denoted by $\text{CB}_{n,2n}$. Given that the adjacency relations of unit interval positroid cells can be described so nicely, we believe the subposet of $\text{CB}_{n,2n}$ consisting of those decorated permutations representing positroids in the closures of unit interval positroid cells of $\text{Gr}^+_{n,2n}$ may have an interesting description. Here we propose a problem stemming from Proposition~\ref{prop:condition for adjacency}.

\begin{prob}
	Describe the subposet of $\emph{CB}_{n,2n}$ consisting of those decorated permutations representing positroids in the closures of unit interval positroid cells of $\emph{Gr}^+_{n,2n}$.
\end{prob}

\section{An Interpretation of the $f$-vector of a Unit Interval Order} \label{sec:f-vector}

In hopes of a more thorough understanding of the \emph{$f$-vectors} of $({\bf 3}+{\bf 1})$-free posets, Skandera and Reed in \cite{SR03} posed the following open problem: characterize the $f$-vectors of unit interval orders. In this aim, we provide a combinatorial interpretation for the $f$-vector of a naturally labeled poset in terms of its antiadjacency matrix. Through this section, $P$ is assumed to be a naturally labeled poset of cardinality $n$ with antiadjacency matrix $A_P = (a_{i,j})$.

\begin{definition}
	The \emph{$f$-vector} of $P$ is the sequence $f = (1, f_0,f_1,\dots,f_{n-1})$, where $f_k$ is the number of $(k+1)$-element chains of $P$.
\end{definition}

We wish to interpret the $k$-element chains of $P$ in terms of some special Dyck paths inside $A_P$. To do this, define a \emph{valley Dyck path} of $A_P$ to be a Dyck path drawn inside $A_P$ that has its endpoints and all its valleys on the main diagonal and all its peaks in positions $(i,j)$ such that $a_{i,j} = 0$. The following figure illustrates a valley Dyck path with three peaks.

\begin{figure}[h]
	\includegraphics[width=4.0cm]{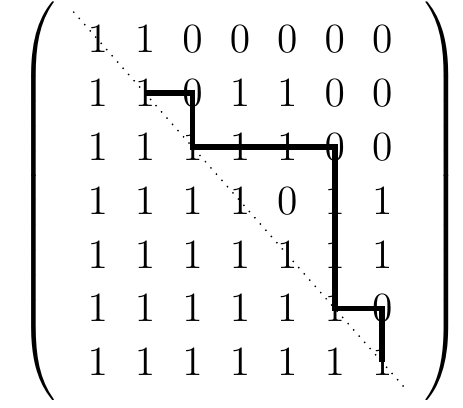}
	\caption{A valley Dyck path with three peaks inside the antiadjacency matrix of the poset displayed in Figure~\ref{fig:three valley Dyck paths}.}
\end{figure}

\begin{prop} \label{prop:f-vector interpretation}
	The entries of the $f$-vector of $P$ are $f_0=n$ and $f_k$ equals the number of valley Dyck paths of $A_P$ having exactly $k$ peaks.
\end{prop}

\begin{proof}
	To each $(k+1)$-element chain $\mathsf{c} : i_1 <_P \dots <_P i_{k+1}$ we can assign a valley Dyck path $v_{\mathsf{c}}$ with $k$ peaks as follows: the $j$-th peak begins at $(i_j,i_j)$, heads east to $(i_j,i_{j+1})$, and then heads south to $(i_{j+1}, i_{j+1})$. To see that $v_{\mathsf{c}}$ is a valley Dyck path, it suffices to notice that every peak of $v_{\mathsf{c}}$ occurs at a zero entry of $A_P$ since $i_j <_P i_{j+1}$ for each $j = 1,\dots,k$. On the other hand, suppose that $v$ is a valley Dyck path with $k$ peaks, namely $(i_1,i'_1), \dots, (i_k,i'_k)$. Then every valley of $v$ is supported on the main diagonal, which means that $i'_j = i_{j+1}$ for each $j=1,\dots,k$. Setting $i_{k+1} = i'_k$, we obtain that $v = v_{\mathsf{c}}$, where $c$ is the $(k+1)$-element chain $i_1 <_P \dots <_P i_{k+1}$. Thus, we have established a bijection that yields the desired result.
\end{proof}

\begin{example}
The naturally labeled poset $P$ depicted below has $f$-vector
\[
	f = (1, 7,12,8,2,0,0,0).
\]
Examples of valley Dyck paths realized on $A_P$ are also shown below.

\begin{figure}[h]
	\centering
		\begin{center}
			\includegraphics[width=2.5cm]{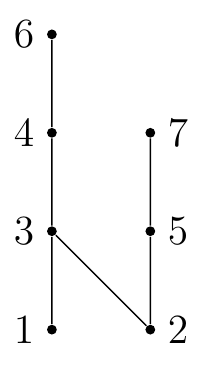}
		\end{center}
		\caption*{Naturally $7$-labeled poset.}
\end{figure}
	\vspace{20pt}
	\begin{figure}[h]
		\begin{center}
			\begin{tabular}{ccc}
				one peak & two peaks & three peaks\\
				\includegraphics[width=3.5cm]{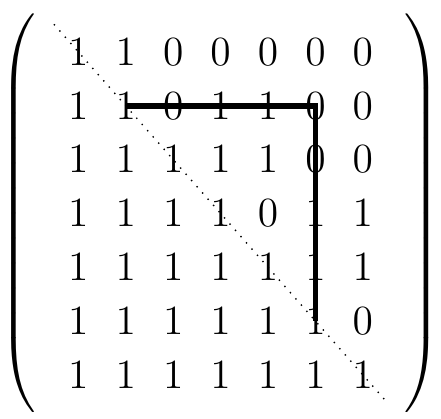} & \includegraphics[width=3.5cm]{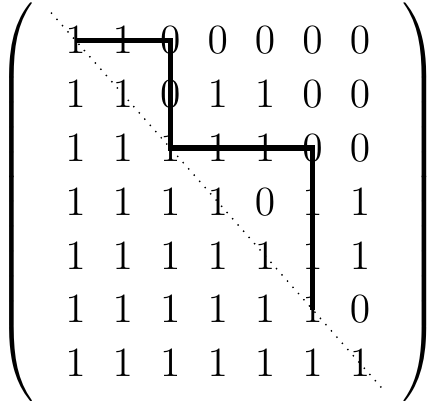} & \includegraphics[width=3.5cm]{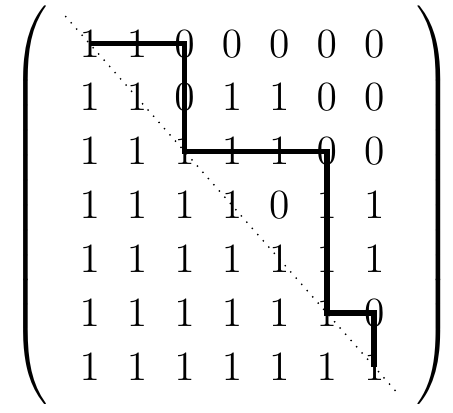}  \\
			\end{tabular}
		\end{center}
		\caption*{Valley Dyck paths inside $A_P$ with one, two, and three peaks.}
	\caption{A naturally labeled poset $P$, along with three distinct valley Dyck paths drawn inside its antiadjacency matrix $A_P$.}\label{fig:dyckposet}
	\label{fig:three valley Dyck paths}
\end{figure}
\end{example}

\noindent {\bf Remark:} Proposition~\ref{prop:f-vector interpretation} provides, in particular, an interpretation of the $f$-vector of any unit interval order. Given that a unit interval order can be labeled so that its antiadjacency matrix is a Dyck matrix, we think that the interpretation of the $f$-vector in Proposition~\ref{prop:f-vector interpretation} might be useful to find an explicit formula for the $f_k$'s. This is because zero and one entries in a Dyck matrix are nicely separated, which could facilitate counting the valley Dyck paths having exactly $k$ peaks.

\begin{prob}
	Given an $n \times n$ Dyck matrix $A$, let $r_i$ be the number of one entries in the $i$-th row of $A$. For $k \in [n-1]$, can we find, in terms of the $r_i$'s, an explicit formula for the number of valley Dyck paths of $A$ containing exactly $k$ peaks?
\end{prob}
\medskip
\medskip

\section*{Acknowledgements}

	While working on this paper, the first author was partially supported by the NSF Alliances for Graduate Education and the Professoriate, while the second author was partially supported by the UC Berkeley Department of Mathematics. Both authors are grateful to Lauren Williams for her guidance throughout this project, Federico Ardila for many helpful conversations, Alejandro Morales for the initial question that motivated this project, and the anonymous referees for many suggestions that led to an improvement of this paper. \\

\end{document}